\documentclass[reqno]{amsart}
\usepackage{amsmath}
\usepackage{arydshln}%
\allowdisplaybreaks[3]
\numberwithin{equation}{section}
\allowdisplaybreaks[3]
\numberwithin{equation}{section}
\numberwithin{equation}{section}

\newtheorem{thm}{\indent Theorem}[section]
\newtheorem{cor}[thm]{\indent Corollary}
\newtheorem{lem}[thm]{\indent Lemma}
\newtheorem{prop}[thm]{\indent Proposition}
\newtheorem{dfn}{{\indent\bf Definition}}[section]
\newtheorem{rmk}{{\indent\bf Remark}}[section]
\newtheorem{expl}{{\indent\bf Example}}[section]

\newcommand{\ra}{\rightarrow}

\newcommand{\mb}{\mbox}

\newcommand{\ol}{\overline}

\newcommand{\ttiny}{\fontsize{5pt}{\baselineskip}\selectfont}
\newcommand{\strl}[2]{\stackrel{\mbox{\ttiny $#1$}}{#2}}

\newcommand{\td}{\tilde}

\newcommand{\fr}{\frac}

\newcommand{\edd}{\end{document}}
\newcommand{\be}{\begin{equation}}
\newcommand{\ee}{\end{equation}}
\newcommand{\bsl}{\backslash}

\newcommand{\lagl}{\langle}
\newcommand{\ragl}{\rangle}
\newcommand{\lmx}{\left(\begin{matrix}}
\newcommand{\rmx}{\end{matrix}\right)}
\newcommand{\ldt}{\left|\begin{matrix}}
\newcommand{\rdt}{\end{matrix}\right|}

\newcommand{\dv}{{\rm div}}

\newcommand{\tr}{{\rm tr\,}}
\newcommand{\vfi}{\varphi}

\newcommand{\veps}{\varepsilon}
\newcommand{\const}{{\rm const}}

\newcommand{\bbr}{{\mathbb R}}
\newcommand{\bbc}{{\mathbb C}}

\newcommand{\ba}{\begin{array}}
\newcommand{\ea}{\end{array}}
\newcommand{\nnm}{\nonumber}
\newcommand{\beal}{\begin{align}}
\newcommand{\eal}{\end{align}}
\newcommand{\bea}{\begin{eqnarray}}
\newcommand{\eea}{\end{eqnarray}}

\newcommand{\spn}{{\rm Span\,}}

\newcommand{\pp}[2]{\fr{\partial #1}{\partial #2}}

\newcommand{\dd}[2]{\fr{d #1}{d #2}}
\begin{document}

\title[Variational characterizations of $\xi$-submanifolds in $\bbr^{m+p}$]{Variational characterizations of $\xi$-submanifolds\\ in the Eulicdean space $\bbr^{m+p}$} 

\author[X. X. Li]{Xingxiao Li$^*$} 

\author[Z. P. Li]{Zhaoping Li} 

\dedicatory{}

%
\thanks{Research supported by
National Natural Science Foundation of China (No. 11671121, No. 11171091 and No. 11371018).}
\address{
School of Mathematics and Information Sciences
\endgraf Henan Normal University \endgraf Xinxiang 453007, Henan, P.R. China
}
\email{xxl$@$henannu.edu.cn}

\address{
School of Mathematics and Information Sciences
\endgraf Henan Normal University \endgraf Xinxiang 453007, Henan, P.R. China
}
\email{li\_zhaoping5555$@$163.com}

\begin{abstract}
$\xi$-submanifold in the Euclidean space $\bbr^{m+p}$ is a natural extension of the concept of self-shrinker to the mean curvature flow in $\bbr^{m+p}$. It is also a generalization of the $\lambda$-hypersurface defined by Q.-M. Cheng et al to arbitrary codimensions. In this paper, some characterizations for $\xi$-submanifolds are established. First, it is shown that a submanifold in $\bbr^{m+p}$ is a $\xi$-submanifold if and only if its modified mean curvature is parallel when viewed as a submanifold in the Gaussian space $(\bbr^{m+p},e^{-\fr{|x|^2}{m}}\lagl\cdot,\cdot\ragl)$; Then, two weighted volume functionals $V_\xi$ and $\bar V_\xi$ are introduced and it is proved that $\xi$-submanifolds can be characterized as the critical points of these two functionals; Also, the corresponding second variation formulas are computed and the ($W$-)stability properties for $\xi$-submanifolds are systematically studied. In particular, it is proved that $m$-planes are the only properly immersed, complete $W$-stable $\xi$-submanifolds with flat normal bundle under a technical condition. It would be interesting if this additional restriction could be removed.
\end{abstract}

\maketitle

\tableofcontents
\section{Introduction}
Let $x: M^m\to\bbr^{m+p}$ be an $m$-dimensional submanifold in the $(m+p)$-dimensional Euclidean space $\bbr^{m+p}$ with the second fundamental form $h$. Then $x$ is called a {\em self-shrinker} to the mean curvature flow if its mean curvature vector field $H:=\tr h$ satisfies
\be\label{eq1.1}
H+x^{\bot}=0,
\ee
where $x^\bot$ is the orthogonal projection of the position vector $x$ to the normal space $T^\bot M^m$ of $x$.

It is well known that the self-shrinker plays an important role in the study of the mean curvature flow. In fact, self-shrinkers correspond to self-shrinking solutions to the mean curvature flow and describe all possible Type I singularities of the flow. Up to now, there have been a plenty of research papers on self-shrinkers together with the asymptotic behavier of the flow. For details of this see, for example,  \cite{a-l}--\cite{c-l2}, \cite{c-p}, \cite{c-z}--\cite{d-x2}, \cite{h}--\cite{l-wei}, \cite{s} and references therein. In particular, the following result well-known (See Corollary \ref{cor3.2} in Section 3):

{\em An immersion $x:M^m\to\bbr^{m+p}$ is a self-shrinker if and only if it is minimal when viewed as a submanifold
of the Gaussian space $(\bbr^{m+p},e^{-\fr{|x|^2}{m}}\lagl\cdot,\cdot\ragl)$.}

In Mar., 2014, Cheng and Wei formally introduced (\cite{c-w2}, finally revised in May, 2015) the definition of $\lambda$-hypersurface of weighted volume-preserving mean curvature flow in Euclidean space, giving a natural generalization of self-shrinkers in the hypersurface case. According to \cite{c-w2}, a hypersurface $x: M^{m}\to\bbr^{m+1}$ is called a $\lambda$-hypersurface if its (scalar-valued) mean curvature $H$ satisfies
\be\label{eq1.2}
H+\langle x, N\rangle=\lambda
\ee
for some constant $\lambda$, where $N$ is the unit normal vector of $x$. They also found some variational characterizations for those new kind of hypersurfaces, proving that {\em a hypersurface $x$ is a $\lambda$-hypersurface if and only if it is the critical point of the weighted area functional $A$ preserving the weighted volume functional $V$} where for any $x_0\in\bbr^{m+1}$ and $t_0\in\bbr$,
$$A(t)=\int_M e^{-\fr{|x(t)-x_0|^2}{2t_0}}d\mu,\quad V(t)=\int_M\lagl x(t)-x_0,N\ragl e^{-\fr{|x(t)-x_0|^2}{2t_0}}d\mu$$
with $N$ the unit normal of $x$. Meanwhile, some rigidity or classification results for $\lambda$-hypersurfaces are obtained, for example, in \cite{c-o-w}, \cite{c-w3} and \cite{g2}; For the rigidity theorems for space-like $\lambda$-hypersurfaces see \cite{l-chp}.

We should remark that this kind of hypersurfaces were also been studied in \cite{mc-r} (arXiv preprint: Jul. 2013; formally published in 2015) where the authors considered the stable, two-sided, smooth, properly immersed solutions to the Gaussian Isoperimetric Problem, namely, they studied hypersurfaces $\Sigma\subset\bbr^{m+1}$ that are second order stable critical points of minimizing the weighted area functional ${\mathcal A}_\mu(\Sigma)=\int_\Sigma e^{-|x|^2/4}d{\mathcal A_\mu}$ for compact (uniformly) normal variations that, in a sense, ``{\em preserve the weighted volume} ${\mathcal V}_\mu(\Sigma)=\int_\Sigma e^{-|x|^2/4}d{\mathcal V_\mu}$''. It turned out that {\em the Euler equation of this variation problem is exactly equivalent to the $\lambda$-hypersurface equation \eqref{eq1.2}}. As the main result, it is also proved that {\em hyperplanes are the only stable ones under the compact normal variations ``preserving the weighted volume''}.

In 2015, the first author and his co-author made a natural generalization of both self-shrinkers and $\lambda$-hypersurfaces, by introducing the concept of $\xi$-submanifolds (\cite{l-ch}, arXiv preprint: 8 Nov. 2015). The main theorem of \cite{l-ch} is a rigidity result of Lagrangian $\xi$-submanifolds in $\bbc^2$, which is motivated by a result of \cite{l-w} for Lagrangian self-shrinkers in $\bbc^2$. By definition, an immersed submanifold $x: M^m\to\bbr^{m+p}$ is called a $\xi$-submanifold if there is a parallel normal vector field $\xi$ such that the mean curvature vector field $H$ satisfies
\be\label{eq1.3}
H+x^{\bot}=\xi.
\ee

We reasonably believe that, if self-shrinkers and $\lambda$-hypersurfaces take places of minimal submanifolds and constant mean curvature hypersurfaces, respectively, then $\xi$-submanifolds are expected to take the place of submanifolds of parallel mean curvature vector. So there would be many properties of $\xi$-submanifolds that are parallel to submanifolds with parallel mean curvature vectors.

In this paper, we aim at giving more characterizations of the $\xi$-submanifolds, especially ones by variation method, the latter being more important since a differential equation usually needs a variational method to solve. For example, self-shrinker equation \eqref{eq1.1} has been exploited a lot by making use of variation formulas. As a main part of this paper, we also study the related stability problems.

The organization of the present paper is as follows:

In Section 2, we present the necessary preliminary material, including some typical examples;

In Section 3 we prove a theorem (Theorem \ref{main1}) which generalizes (to $\xi$-submanifolds) a well-known result that self-shrinkers are equivalent to minimal submanifolds in the Gaussian space;

In Section 4, we introduce, for a given manifold $M^m$ of dimension $m$, two families of weighted volume functionals $V_\xi$ and $\bar V_\xi$ in \eqref{vxi} parametrized by $\bbr^{m+p}$-valued functions $\xi:M^m\to\bbr^{m+p}$. Then we compute the first variation formulas (Theorem \ref{main2}) which give that $\xi$-submanifolds are exactly the critical points of $V_\xi$ and $\bar V_\xi$ with $\xi$ suitably chosen (Corollary \ref{cor1}). We also compute the second variation formula of both functionals for $\xi$-submanifolds, in such a situation $V_\xi$ and $\bar V_\xi$ being essential the same (Theorem \ref{main3}).

In Section 5, Section 6 and Section 7, we study the stability problem of $\xi$-submanifolds. After checking that all the canonical examples are not stable in the usual sense (Section 5), we introduce in Section 6 the concept of $W$-stability and are able to prove that, among the typical examples given in Section 2, only the $m$-planes are $W$-stable (Theorem \ref{main4} and Theorem \ref{prop1'}). Meanwhile we give an index estimate for the standard sphere (Theorem \ref{prop1'}).

Finally in the last section (Section 7) we are able to prove the following main Theorem:

{\bf Theorem \ref{main6}} {\em Let $x:M^m\to\bbr^{m+p}$ be a properly immersed, complete and $W$-stable $\xi$-submanifold with flat normal bundle satisfying  \be\label{A} h(A_\xi(x^\top),v^\top)=0,\quad\forall\,v\in\bbr^{m+p},\ee
where $A_\xi$ denotes the Weingarten map in the direction of $\xi$. Then $x(M^m)$ must be an $m$-plane.}

Then the following corollary is direct:

{\bf Corollary} (Corollaries \ref{cor7.1} and \ref{cor7.2}) {\em Any properly immersed, complete and $W$-stable $\xi$-submanifold in $\bbr^{m+p}$ with flat normal bundle must be an $m$-plane if the Weigarten map $A_\xi$ with respect to $\xi$ vanishes.

In particular, Any properly immersed, complete and $W$-stable self-shrinker in $\bbr^{m+p}$ with flat normal bundle must be an $m$-plane.}

Consequently, the following problem is interesting:

{\bf Problem}: {\em Naturally we believe and expect that the additional condition \eqref{A} in Theorem \ref{main6} could be dropped; Furthermore, motivated by the main theorem of \cite{mc-r}, it is also expected, without any additional conditions, that the $m$-planes are the only properly immersed, complete $W$-stable $\xi$-submanilds or, if it is not the case, more examples could be found.}

{\rmk\em Our discussion of variation problem for $\xi$-submanifolds naturally gives a motivation of variational characterization of the submanifolds with parallel mean curvature vectors in the Euclidean space. For the detail of this, see Remark \ref{rmk} at the end of Section 5. 

Furthermore, by using an explanation of the $VP$-variation with some kind of related $(m+1)$-dimensional volume enclosed by a compact $\xi$-submanifold, the isoparametric problem for submanifolds of higher codimension will be considered elsewhere in a forthcoming paper (\cite{li-li}).}

{\bf Acknowledgement} This research is supported by National Natural Science Foundation of China (No. 11671121, No. 11171091 and No. 11371018). The first author thanks Professor D. T. Zhou for kindly introducing to him the reference \cite{mc-r}.

\section{$\xi$-sumanifolds--definition and typical examples}

Let $\bbr^{m+p}$ be the $m$-dimensional Euclidean space with the standard metric denoted by $\lagl\cdot,\cdot\ragl$ and $x:M^m\to \bbr^{m+p}$ be an immersion with the induced metric $g$, the second fundamental form $h$ and the mean curvature vector $H:=\tr_g h$. Denote by $TM$ the tangent space of $M$ and define $T^\bot M:=(x_*(TM))^\bot$ to be the normal space of $x$ in $\bbr^{m+p}$.

\begin{dfn}[$\xi$-submanifolds, \cite{l-ch}] The immersed submanifold $x:M^m\to \bbr^{m+p}$ is called a $\xi$-submanifold if the normal vector field $H+x^\bot$ is parallel in $T^\bot M$, or the same, there exists some parallel normal field $\xi\in \Gamma(T^\bot M)$ such that
\be\label{xisub}
H+x^\bot=\xi.
\ee
\end{dfn}

Clearly, self-shrinkers of the mean curvature flow are a special kind of $\xi$-submanifolds.

The following are some typical examples of $\xi$-submanifolds:

\begin{expl}[The $\xi$-curves]\label{expl} \rm\mb{}

Let $x:(a,b)\to\bbr^{1+p}$ be a unit-speed smooth curve (that is, with an arc-length parameter $s$). Denote by $\{T,e_\alpha:\ 2\leq\alpha\leq 1+p\}$ the Frenet frame with $T:=\dot x\equiv\pp{x}{s}$ being the unit tangent vector,  and $\kappa_i$ the $i$-th curvature, $i=1,\cdots, p$. Then we have the following Frenet formula:
\be\label{fre}\dot T=\kappa_1e_2,\ \dot e_2=-\kappa_1T+\kappa_2e_3,\ \cdots,\  \dot e_p=-\kappa_{p-1}e_{p-1}+\kappa_pe_{p+1},\ \dot e_{1+p}=-\kappa_pe_p.\ee
In particular, if there exists some $i$ such that $\kappa_i\equiv 0$, then it must hold that $\kappa_j\equiv 0$ for all $j>i$. Sometimes we call $\kappa:=\kappa_1$ and $\tau:=\kappa_2$ the curvature and the (first) torsion of $x$. Now the definition equation \eqref{xisub} becomes
$\left(\dd{}{s}(\dot T+x-\lagl x,T\ragl T)\right)^\bot\equiv 0$ which, by \eqref{fre}, is equivalent to
\be\dot\kappa_1-\kappa_1\lagl x,T\ragl\equiv 0,\quad \kappa_1\kappa_2\equiv 0.\label{xisub1}\ee
It follows that

{\em $x$ is a $\xi$-curve if and only if it is a plane curve with the curvature $\kappa$ satisfying}
\be\label{curve}\dot\kappa-\kappa\lagl x,\dot x\ragl\equiv 0.\ee
In particular,

{\em $x$ is a self-shrinker if and only if it is a plane curve with the curvature $\kappa$ satisfying}
\be\label{sscurve}\kappa_r+\lagl x,N\ragl\equiv 0,\ee
where $\kappa_r$ is the relative curvature and $N:=\pm e_2$ is the unit normal of $x$ pointing the left of $T$. Note that curves in the plane satisfying \eqref{sscurve} are classified by U. Abresch and J. Langer in \cite{a-l} which are now known as {\em Abresch-Langer curves} (see \cite{l-w}).
\end{expl}

\begin{expl}[The $m$-planes not necessarily passing through the origin]\label{expl1} \rm\mb{}

An $m$-plane $x:P^m\to\bbr^{m+p}$ ($p\geq 0$) is by definition the inclusion map of a $m$-dimensional connected, complete and totally geodesic submanifold of $\bbr^{m+p}$. In other words, those $P^m$s are subplanes of dimension $m$ in $\bbr^{m+p}$ that are not necessarily passing through the origin. Let $p_0$ be the orthogonal projection of the origin $0$ onto $P^m$ and $\xi$ be the position vector of $p_0$ which is constant and is thus parallel along $P^m$. Clearly $P^m$ is a $\xi$-submanifold because $H\equiv 0$ and the tangential part $x^\top$ of $x$ is precisely $x-\xi$.
\end{expl}

\begin{expl}[The standard spheres centered at the origin]\label{expl2} \rm\mb{}

For a given point $x_0\in\bbr^{m+1}$ and a positive number $r$. Define
$$
S^m(r,x_0)=\{x\in \bbr^{m+1};\ |x-x_0|=r\},
$$
the standard $m$-sphere in $\bbr^{m+1}$ with radius $r$ and center $x_0$. In particular, we denote $S^m(r):=S^m(r,0)$. It is easily find that $S^m(r,x_0)$ is a $\xi$-submanifold if and only if $x_0=0$.

In fact, since $x-x_0$ is a normal vector field of length $r$, the normal part $x^\bot$ of $x$ is
$$
x^\bot=\fr1{r^2}\lagl x,x-x_0\ragl(x-x_0).
$$
Note that $H=-\fr m{r^2}(x-x_0)$ is parallel. It follows that $H+x^\bot$ is parallel if and only if $x^\bot$ is. This is clearly equivalent to that $\lagl x,dx\ragl\equiv 0$ which is true if and only if $x_0=0$.
\end{expl}

\begin{expl}[Submanifolds in a sphere with parallel mean curvature vector]\label{expl3}
\rm\mb{}

Let $x:M^m\to S^{m+p}(a)\subset \bbr^{m+p+1}$ be a submanifold in the standard sphere $S^{m+p}(a)$ of radius $a$, which is of parallel mean curvature vector $H$. Then as a submanifold of $\bbr^{m+p+1}$, $x$ is a $\xi$-submanifold.\end{expl}

In fact, as the submanifold of $\bbr^{m+p+1}$, the mean curvature vector of $x$ is $\bar H=\triangle x=H-\fr{m}{a^2}x$. Thus $\xi:=\bar H+x^\bot=H+(1-\fr{m}{a^2})x$ which is clearly parallel. In particular, $x(M^m)\subset\bbr^{m+p+1}$ is a self-shrinker if and only if $x(M^m)\subset S^{m+p}(a)$ is a minimal submanifold.

\begin{expl}[The product of $\xi$-submanifolds]\label{expl4} \rm\mb{}

Let $x_a:M^{m_a}\to\bbr^{m_a+p_a}$, $a=1,2$, be two immersed submanifolds. Denote $m=m_1+m_2$, $p=p_1+p_2$ and $M^m=M^{m_1}\times M^{m_2}$. Then it is not hard to show that $x:=x_1\times x_2:M^m\to\bbr^{m+p}$ is a $\xi$-submanifold if and only if both $x_1$ and $x_2$ are $\xi$-submanifolds.

In particular, for any given positive numbers $r_1,\cdots,r_k$ ($k\geq 0$), positive integers $m_1,\cdots, m_k,n_1,\cdots,n_l$ ($l\geq 0$, $k+l>0$) and $n\geq n_1+\cdots+n_l$, the embedding
\be\label{product}
x:S^{m_1}(r_1)\times\cdots\times S^{m_k}(r_k)\times P^{n_1}\times\cdots\times P^{n_l}\to\bbr^{m_1+\cdots+m_k+k+n}
\ee
are all $\xi$-submanifolds.
\end{expl}

\section{As submanifolds of the Gaussian space}

As mentioned in the introduction, the $m$-dimensional self-shrinkers of the mean curvature flow in the Euclidean space $\bbr^{m+p}\equiv (\bbr^{m+p},\lagl\cdot,\cdot \ragl)$ is equivalent to being the minimal submanifolds when viewed as submanifolds in the Gaussian metric space $(\bbr^{m+p}, \bar g)$ where $\bar g:= e^{-\fr{|x|^2}{m}}\lagl\cdot,\cdot \ragl$. In this section, we generalize this to $\xi$-submanifolds to obtain our first characterization. In fact, we will prove a theorem which says that $\xi$-submanifolds are essentially equivalent to being submanifolds of parallel mean curvature in $(\bbr^{m+p}, \bar g)$.

For an immersion $x:M^m\to\bbr^{m+p}$, we use $(\ol\cdots)$ to denote geometric quantities when $x$ is taken as  an immersion into $(\bbr^{m+p}, \bar g)$ that correspond those quantities $(\cdots)$ when $x$ is taken as an immersion into $(\bbr^{m+p}, \lagl\cdot,\cdot \ragl)$. So, for example, we have the induced metric $\bar g$, the second fundamental form $\bar h$ and the mean curvature $\bar H$, etc. To make things more clear, we would like to introduce a ``{\em modified mean curvature}'' for the immersion $x$, which is defined as $\td H=e^{-\fr{|x|^2}{2m}}\bar H$.

Now our first characterization theorem can be stated as follows:

{\thm[The first characterization]\label{main1} An immersion $x:M^m\to\bbr^{m+p}$ is a $\xi$-submanifold if and only if it is of parallel modified mean curvature $\td H$.}

\begin{proof}
Denote by $D$ and $\bar D$ the Levi-Civita connections of $(\bbr^{m+p}, \lagl\cdot,\cdot\ragl)$ and $(\bbr^{m+p}, \bar g)$ with $\bar g= e^{-\fr{|x|^2}{m}}\lagl\cdot,\cdot\ragl)$, respectively.  For any given frame field $\{e_A;\ A=1,2\cdots,m+p\}$, the corresponding connection coefficients of $D$ and $\bar D$ are respectively denoted by $\Gamma^C_{AB}$ and $\bar\Gamma^C_{AB}$, where we assume that $A,B,C,\cdots=1,2,\cdots m+p$. Then by a easy computation using the Koszul formula we can find
\be\label{Gamma}
\bar\Gamma^C_{AB}=\Gamma^C_{AB}+\fr1m\left(g(x,e_D)g_{AB}g^{CD}-g(x,e_A)\delta^C_B-g(x,e_B)\delta^C_A\right),
\ee
or equivalently,
\be\label{D}
\bar D_{e_B}e_A=D_{e_B}e_A+\fr1m(g_{AB}x-g(x,e_A)e_B-g(x,e_B)e_A).
\ee

Now given an immersion $x:M^m\to\bbr^{m+p}$, the induced metric on $M^m$ by $x$ of the ambient metric $\bar g$ will still be denoted by $\bar g$. Choose a frame field $\{e_i,e_\alpha\}$ along $x$ such that $e_i$, $i=1,2,\cdots,m$, are tangent to $M^m$ and $e_\alpha$, $\alpha=m+1,\cdots,m+p$ are normal to $x_*(TM^m)$ satisfying $\lagl e_\alpha,e_\beta\ragl\equiv g(e_\alpha,e_\beta)=\delta_{\alpha\beta}$. Then by the Gauss formula and \eqref{Gamma} or \eqref{D}, we find the relation between the second fundamental forms $\bar h$ and $h$ is as follows:
\be\label{h}
\bar h_{ij}\equiv \bar h(e_i,e_j)=\left(\bar D_{e_j}e_i\right)^\bot=h_{ij}+\fr1m x^\bot g_{ij}
\ee
where $h_{ij}=h(e_i,e_j)=\left(D_{e_j}e_i\right)^\bot$. It follows that the mean curvature vectors satisfy
\be\label{H}
\bar H\equiv\bar g^{ij}\bar h_{ij}=e^{\fr{|x|^2}{m}}(H+x^\bot).
\ee

Now we compute the covariant derivative of the modified mean curvature $\td H\equiv e^{-\fr{|x|^2}{2m}}\bar H$ with respect to the normal connection $\bar D^\bot$. First we note that, since $\bar g$ is conformal to $\lagl\cdot,\cdot\ragl$ on $\bbr^{m+p}$, $\{e_\alpha\}$ which satisfies $\lagl e_\alpha,e_\beta\ragl=\delta_{\alpha\beta}$ remains a normal frame field of $x$ considered as the immersion into $(\bbr^{m+p},\bar g)$, of course not orthonormal anymore. Thus we can write
$$
\td H=\sum\td H^\alpha e_{\alpha}\text{\ with\ } \td H^\alpha=e^{\fr{|x|^2}{2m}}(H^\alpha+\lagl x,e_\alpha\ragl)
$$
where $H=\sum H^\alpha e_\alpha$. Note that by \eqref{Gamma},
$$
\bar\Gamma^\alpha_{\beta i}=\Gamma^\alpha_{\beta i}-\fr1m\lagl x,e_i\ragl\delta^\alpha_\beta, \quad\forall \alpha,\beta, i.
$$
It follows that, for each $\alpha=m+1,\cdots,m+p$,
\begin{align*}
\left(\bar D^\bot_{e_i}\td H\right)^\alpha=&e_i(\td H^\alpha)+\td H^\beta\bar\Gamma^\alpha_{\beta i}\\
=&e_i\big(e^{\fr{|x|^2}{2m}}\big)(H^\alpha+\lagl x,e_\alpha\ragl)+e^{\fr{|x|^2}{2m}}(e_i(H^\alpha)+e_i\lagl x,e_\alpha\ragl)\\
&+e^{\fr{|x|^2}{2m}}(H^\beta+\lagl x,e_\beta\ragl)(\Gamma^\alpha_{\beta i}-\fr1m\lagl x,e_i\ragl\delta^\alpha_\beta)\\
=&e^{\fr{|x|^2}{2m}}(e_i(H^\alpha)+e_i\lagl x,e_\alpha\ragl+H^\beta\Gamma^\alpha_{\beta i}+\lagl x,e_\beta\ragl\Gamma^\alpha_{\beta i})\\
=&e^{\fr{|x|^2}{2m}}\big(D^\bot_{e_i}(H+x^\bot)\big)^\alpha,
\end{align*}
where $\bar D^\bot$, $D^\bot$ denote the induced normal connections accordingly. Thus Theorem \ref{main1} is proved.
\end{proof}

The following conclusion is direct by \eqref{H}:

{\cor\label{cor3.2} An immersion $x:M^m\to\bbr^{m+p}$ is a self-shrinker if and only if it is minimal when viewed as a submanifold of the Gaussian space $(\bbr^{m+p},\bar g)$.}

\section{Variational characterizations}

In this section, we first define two functionals and derive the corresponding first and second variation formulas, aiming to establish variational characterizations of the $\xi$-submanifolds.

For a given manifold $M\equiv M^m$ of dimension $m$, define
$$
{\mathcal M}:=\{\text{all the immersions }x:M^m\to \bbr^{m+p}\}
$$
and let $\xi:M^m\to\bbr^{m+p}$ be a vector-valued function on the manifold $M^m$. Then we can naturally introduce as follows two kinds of interesting functionals $V_\xi$ and $\bar V_\xi$ on $\mathcal M$ which are parametrized by $\xi$:
\be\label{vxi}
V_\xi(x):=\int_M e^{-f_x}dV_x,\quad \bar V_\xi(x)=\int_M e^{-\bar f_x}dV_x,
\ee
where for any $p\in M^m$, $f_x(p):=\frac{1}{2}|x(p)-\xi(p)|^2$, $\bar f_x(p)=f_x(p)-\fr12|\xi(p)|^2$ and $dV_x$ is the volume element of the induced metric $g_x$ of $x$.

\begin{rmk}\rm (1) These two functionals $V_\xi$ and $\bar V_\xi$ are both of weighted volumes in a sense since, for example,  the weighted volume element $e^{-\frac{1}{2}|x-\xi|^2} dV_x$ corresponding to the first one can be viewed as induced from an unnormalized ``general Gaussian measure'' on the ambient Euclidean space $\bbr^{m+p}$ with ``{\em mean}" $\xi$. Note that when $\xi$ is constant as in the case of $m$-planes, $\left(\fr1{\sqrt{2\pi}}\right)^{m+p}e^{-f}dV_{\bbr^{m+p}}$ is nothing but the usual general Gaussian measure with the mean $\xi$ (and the {\em variance} $\sigma^2\equiv 1$)\,\footnote{See the explanation in Wikipedia, the free encyclopedia under the title ``Gaussian measure''}; Meanwhile the functional $\bar V_\xi$ is clearly a new weighted volume obtained from $V_\xi$ by just adding a new weight $e^{\fr12|\xi|^2}$. Also, the weight-function $e^{-f}$ or $e^{-\bar f}$ naturally has a close relation with the definition of the {\em Hermitian Polynomials} (see, for example, G. Dattoli, A. Torre, S. Lorenzutta, G. maino and C. Chiccoli, {\em Multivariable Hermite Polynormials and Phase-Space Dynamics}, on the website: http://ntrs.nasa.gov/archive/nasa/casi.ntrs.nasa.gov/19950007516.pdf).  These polynomials will also be used later in our stability discussion in Section 5.

(2) All of the canonical $\xi$-submanifolds (that is, $m$-planes $P^m$, standard $m$-spheres $S^m(r)$) and their products \eqref{product} have finite values for both the functionals $V_\xi$ and $\bar V_\xi$, where $\xi$ is chosen to be $H+x^\bot$.\end{rmk}

Now let $x\in {\mathcal M}$ be fixed with the induced Riemannian metric $g:=x^*\lagl\cdot,\cdot\ragl$ and suppose that $F:M\times(-\varepsilon,\varepsilon)\to\bbr^{m+p}$ is a
variation of $x$ with $\eta:=F_*(\pp{}{t})|_{t=0}$ being the the corresponding variation vector field.
For $p\in M$, $t\in(-\veps,\veps)$, denote
$$
x_t(p)=F(p,t),\quad\pp{F}{t}=F_*\left(\pp{}{t}\right),\quad \pp{F}{u^i}=F_*(\pp{}{u^i})\equiv (x_t)_*\left(\pp{}{u^i}\right)
$$
where $(u^i)$ is a local coodinates on $M$. We always assume that, for each $t\in (-\veps,\veps)$, $x_t:M^m\to\bbr^{m+p}$ is an immersion, that is, $x_t\in {\mathcal M}$, $t\in(-\veps,\veps)$.

{\dfn[Compact variation]\rm A variation $F:M\times(-\varepsilon,\varepsilon)\to\bbr^{m+p}$ is called compactly supported, or simply {\em compact}, if there exists a relatively compact open domain $B$ such that, for each $t\in(-\veps,\veps)$, the support set $\ol{\{p\in M^m;\ \pp{F}{t}(p)\neq 0\}}$ of the vector field $\pp{F}{t}$ is contained in $B$.}

Denote $f_t=f_{x_t}$, $\bar f_t=\bar f_{x_t}$ and
$$
\Gamma_0(T^\bot(M))=\{\text{all smooth normal vector fields $\eta$ of $x$ with compact support}\}.
$$

\begin{thm}[The first variation formula]\label{main2} Let $F$ be a compact variation of $x$. Then
\begin{align}
V_\xi'(t)=&-\int_M
\lagl(H_{t}+x_{t}^\bot-\xi)+\nabla^t\left(\lagl x_t,\xi\ragl-\fr12|\xi|^2\right),\pp{F}{t}\ragl e^{-f_t}dV_t,\label{1vf}\\
\bar V_\xi'(t)=&-\int_M
\lagl(H_{t}+x_{t}^\bot-\xi)+\nabla^t\lagl x_t,\xi\ragl,\pp{F}{t}\ragl e^{-\bar f_t}dV_t,\label{1wf}
\end{align}
where $H_t$ is the mean curvature vector of the immersion $x_t$, $\nabla^t$ is the gradient operator of the induced metric $g_{x_t}$ and $dV_t=dV_{x_t}$.

In particular, if $F$ is a normal variation of $x$, that is, $\eta\in \Gamma_0(T^\bot(M))$, then
\begin{align}
V_\xi'(0)=&-\int_M\lagl(H+x^\bot-\xi),\eta\ragl e^{-f}dV,\label{vf0}\\
\bar V_\xi'(0)=&-\int_M\lagl(H+x^\bot-\xi),\eta\ragl e^{-\bar f}dV.\label{wf0}
\end{align}
\end{thm}

\begin{proof} For simplicity, we shall always write $f=f_t$ in the computation.
It is well known that
\begin{align*}
\pp{}{t}dV_t=&\left(\dv\left(\pp{F}{t}\right)^\top -\lagl H_t,\pp{F}{t}\ragl\right)dV_t\\
=&\left(\left(g_t^{ij}\lagl\pp{F}{u^i},\pp{F}{t}\ragl\right)_{,j}-\lagl H_t,\pp{F}{t}\ragl\right)dV_t.
\end{align*}
Furthermore
$$
\pp{}{t}e^{-f}=-e^{-f}\pp{f}{t}=-e^{-f}\lagl x_t-\xi,\pp{F}{t}\ragl.
$$
Thus by using the divergence theorem, we find
\begin{align*}
V_\xi'(t)=&\int_M\pp{}{t}\left(e^{-f}dV_t\right)=\int_M\left(\pp{}{t}e^{-f})dV_t+e^{-f}\pp{}{t}dV_t\right)\\
=&\int_M\left(-e^{-f}\lagl x_t-\xi,\pp{F}{t}\ragl  +e^{-f}\left(\left(g_t^{ij}\lagl\pp{F}{u^i},\pp{F}{t}\ragl\right)_{,j}-\lagl H_t,\pp{F}{t}\ragl\right)\right) dV_t\\
=&-\int_M\left(\lagl H_t+ x_t^\bot-\xi,\pp{F}{t}\ragl +g_t^{ij}\pp{}{u^j}\left(\lagl x_t,\xi\ragl-\fr12|\xi|^2\right)\pp{F}{u^i},\pp{F}{t}\ragl \right)e^{-f}dV_t\\
=&-\int_M\left(\lagl (H_t+ x_t^\bot-\xi)+\nabla^t\left(\lagl x_t,\xi\ragl-\fr12|\xi|^2\right),\pp{F}{t}\ragl \right)e^{-f}dV_t,
\end{align*}
which gives \eqref{1vf}. The other formula \eqref{1wf} is derived in the same way.
\end{proof}

\begin{cor}[Variational characterizations]\label{cor1}
An immersion $x\in{\mathcal M}$ is a $\xi$-submanifold if and only if there exists a parallel normal vector field $\xi\in\Gamma(T^\bot M)$ such that $x$ is the critical point of both the functionals $V_\xi$, $\bar V_\xi$ for all the compact normal variations of $x$.
\end{cor}

To find the second variational formulas, we suppose that $x$ is a $\xi$-submanifold, that is, $H+x^\bot=\xi $, where $\xi$ is a parallel normal vector of $x$.  In particular, $|\xi|^2$ is a constant. Note that in this case, the two functionals $V_\xi$ and $\bar V_\xi$ are essentially the same. So in the argument that follows we only need to consider $V_\xi$.

Suppose that $F$ is a compact normal variation of $x$. Then

\begin{align}
V_\xi''(0)=&-\int_M
\lagl D_{\pp{}{t}}\left((H_{t}+x_{t}^\bot-\xi)+\nabla^t\lagl x_t,\xi\ragl\right),\pp{F}{t}\ragl |_{t=0}e^{-f}dV\nnm\\
&-\int_M
\lagl \nabla^t\lagl x_t,\xi\ragl, D_{\pp{}{t}}\pp{F}{t}\ragl|_{t=0} e^{-f}dV\nnm\\
=&-\int_M
\lagl D_{\pp{}{t}}\left((H_{t}+x_{t}^\bot-\xi)+\nabla^t\lagl x_t,\xi\ragl\right)|_{t=0},\eta\ragl e^{-f}dV\nnm\\
&-\int_M
\lagl\nabla\lagl x,\xi\ragl, D_{\pp{}{t}}\pp{F}{t}|_{t=0}\ragl e^{-f}dV.
\end{align}
Since
$$H_{t}=(g_{t})^{ij}h_t(\pp{}{u^i},\pp{}{u^j})=(g_{t})^{ij}\left( D_{\pp{}{u^j}}(x_{t})_{*}\pp{}{u^i}-(x_{t})_{*}\nabla_{\pp{}{u^j}}^{t}\pp{}{u^i}\right),$$
we have
\begin{align}
 D_{\pp{}{t}}H_{t}=\pp{}{t}(g_{t})^{ij}h_t(\pp{}{u^i},\pp{}{u^j})
+(g_{t})^{ij} D_\pp{}{t}\left( D_{\pp{}{u^j}}(x_{t})_{*}\pp{}{u^i}-(x_{t})_{*}\nabla_{\pp{}{u^j}}^{t}\pp{}{u^i}\right).
\end{align}

On the other hand
\begin{align*}
\left(\pp{}{t}(g_{t})^{ij}\right)|_{t=0}=&-\left(\left((g_{t})^{ik}(g_{t})^{jl}\lagl D_{\pp{}{t}}\pp{F}{u^k},\pp{F}{u^l}\ragl+
\lagl \pp{F}{u^k}, D_{\pp{}{t}}\pp{F}{u^l}\ragl \right) \right)|_{t=0}\\
=&-g^{ik}g^{jl}\left(\pp{}{u^k}\lagl \pp{F}{t},\pp{F}{u^l}\ragl-\lagl \pp{F}{t} , D_{\pp{}{u^k}}\pp{F}{u^l}\ragl\right)|_{t=0}\\&-g^{ik}g^{jl}\left(\pp{}{u^l}\lagl \pp{F}{t},\pp{F}{u^k}\ragl-\lagl \pp{F}{t} , D_{\pp{}{u^l}}\pp{F}{u^k}\ragl\right)|_{t=0}\\
=&g^{ik}g^{jl}\lagl h(\pp{}{u^k},\pp{}{u^l}),\eta\ragl+g^{ik}g^{jl}\lagl h(\pp{}{u^l},\pp{}{u^k}),\eta\ragl,
\end{align*}
and by the flatness of $\bbr^{m+p}$,
\begin{align*}
 D_\pp{}{t} D_{\pp{}{u^j}}\pp{F}{u^i}|_{t=0} =&D_\pp{}{u^j} D_\pp{}{t}(x_t)_*\pp{}{u^i}
+ D_{[\pp{}{t},\pp{}{u^j}]}(x_t)_*\pp{}{u^i}|_{t=0}\\
=&D_{\pp{}{u^j}}(D_\pp{}{u^i}^\bot\eta
-x_*(A_\eta \pp{}{u^i}))\\
=&D_{\pp{}{u^j}}^\bot D_\pp{}{u^i}^\bot\eta
-h(\pp{}{u^j},A_\eta (\pp{}{u^i}))\\
&-x_{*}(A_{D_{\pp{}{u^i}}^\bot\eta}\pp{}{u^j})-x_{*}(\nabla_{\pp{}{u^j}}(A_\eta \pp{}{u^i}))
\end{align*}
where $A_\eta$ is the Weingarten operator of $x$ with respect to the variation vector $\eta$.
Moreover
\begin{align*}
 D_\pp{}{t}((x_{t})_{*}\nabla_{\pp{}{u^j}}^{t}\pp{}{u^i})|_{t=0}
=& D_\pp{}{t}((\Gamma_{t})_{ij}^{k}(x_{t})_{*}\pp{}{u^k})|_{t=0}\\
=&\pp{}{t}((\Gamma_{t})_{ij}^{k})|_{t=0}x_*\pp{}{u^k}+\Gamma_{ij}^{k} D_\pp{}{t}\left(\pp{F}{u^k}\right)|_{t=0}\\
=&\pp{}{t}((\Gamma_{t})_{ij}^{k})|_{t=0}x_*\pp{}{u^k}+D_{\nabla_{\pp{}{u^j}}\pp{}{u^i}}\eta.
\end{align*}
It then follows that
\begin{align}
&\lagl\pp{}{t}(g_{t})^{ij} h_t(\pp{}{u^i},\pp{}{u^j})|_{t=0},\eta\ragl
=2g^{ik}g^{jl}\lagl h({\pp{}{u^k},{\pp{}{u^l}}}),\eta\ragl \lagl h({\pp{}{u^i},{\pp{}{u^j}}}),\eta\ragl,\\
&g^{ij}\lagl D_\pp{}{t} D_{\pp{}{u^j}}\pp{F}{u^i}|_{t=0},\eta\ragl
=g^{ij}\left(\lagl D_{\pp{}{u^j}}^\bot D_\pp{}{u^i}^\bot\eta
-h(\pp{}{u^j},A_\eta (\pp{}{u^i})),\eta\ragl\right),\\
&g^{ij}\left(\lagl D_\pp{}{t}((x_{t})_{*}\nabla_{\pp{}{u^j}}^{t}\pp{}{u^i})|_{t=0},\eta\ragl\right)=g^{ij}\lagl D_{\nabla_{\pp{}{u^j}}\pp{}{u^i}}^\bot \eta,\eta\ragl.
\end{align}
Hence
\begin{align*}
\lagl  D_\pp{}{t}H_{t}|_{t=0},\eta\ragl
=&\lagl g^{ij}(D_{\pp{}{u^i}}^\bot D_\pp{}{u^j}^\bot\eta-D_{\nabla_{\pp{}{u^i}}\pp{}{u^j}}^\bot \eta),\eta\ragl\\
&+g^{ik}g^{jl}\lagl h(\pp{}{u^k},\pp{}{u^l}),\eta\ragl \lagl h(\pp{}{u^i},\pp{}{u^j}),\eta\ragl\\
=&\lagl \triangle_{M}^\bot\eta ,\eta\ragl+g^{ik}g^{jl}\lagl h(\pp{}{u^k},\pp{}{u^l}),\eta\ragl \lagl h(\pp{}{u^i},\pp{}{u^j}),\eta\ragl\\
=&\lagl \triangle_{M}^\bot\eta+g^{ik}g^{jl}\lagl h_{ij},\eta\ragl h_{kl} ,\eta\ragl,
\end{align*}
where $h_{ij}=h(\pp{}{u^i},\pp{}{u^j})$. Furthermore,
\begin{align*}
&\lagl  D_\pp{}{t} (x_{t}^\bot-\xi)|_{t=0},\eta\ragl
=\lagl D_\pp{}{t} x_{t}|_{t=0}- D_\pp{}{t}(x_t)^\top|_{t=0},\eta\ragl\\
=&\lagl \eta,\eta\ragl-\lagl D_\pp{}{t} \left((g_t)^{ij}\lagl x_{t},\pp{F}{u^i}\ragl\pp{F}{u^j}\right) |_{t=0},\eta\ragl\\
=&\lagl \eta,\eta\ragl-\lagl  D_{x^\top}\eta,\eta\ragl=\lagl \eta,\eta\ragl-\lagl D^\bot_{x^\top}\eta,\eta\ragl.
\end{align*}
Therefore
\begin{align*}
\lagl D_{\pp{}{t}}(H_{t}+x_{t}^\bot-\xi),\pp{F}{t}\ragl|_{t=0}=\lagl \triangle_{M}^\bot\eta-D^\bot_{x^\top}\eta+g^{ik}g^{jl}\lagl h_{ij},\eta\ragl h_{kl} +\eta,\eta\ragl
\end{align*}

Meanwhile,
\begin{align*}
\lagl D_{\pp{}{t}}(\nabla^t\lagl x_t,\xi\ragl)|_{t=0},\eta\ragl
=&\lagl (g_t)^{ij}\pp{}{u^i}\lagl x_t,\xi\ragl D_{\pp{}{t}}\pp{F}{u^i}|_{t=0},\eta\ragl\\
=&\lagl g^{ij}\pp{}{u^i}\lagl x,\xi\ragl D_{\pp{}{u^i}}\eta,\eta\ragl=\lagl D^\bot_{\nabla\lagl x,\xi\ragl}\eta,\eta\ragl\\
=&-\lagl D^\bot_{A_\xi(x^\top)}\eta,\eta\ragl
\end{align*}
since $\xi$ is parallel along $x$.

Note that the ambient space $\bbr^{m+p}$ is flat and $|\xi|^2$ is constant on $M^m$. Thus, by summing up, we have proved the following second variation formulas for $\xi$-submanifolds:

\begin{thm}\label{main3} Let $x:M^m\to\bbr^{m+p}$ be a $\xi$-submanifold. Then for any compact normal variation $F:M^m\times(-\veps,\veps)\to\bbr^{m+p}$ we have
\begin{align}
V_{\xi}^{''}(0)=&-\int_M\Big(\lagl\triangle_{M}^\bot(\eta) ,\eta\ragl-\lagl D^\bot_{x^\top+A_\xi(x^\top)}\eta+g^{ik}g^{jl}\lagl h_{ij},\eta\ragl h_{kl}+\eta,\eta\ragl\nnm\\
&+\lagl \nabla \lagl x,\xi\ragl, D_{\pp{}{t}}\pp{F}{t}|_{t=0}\ragl\Big)e^{-f}dV,\label{2vf}\\
\bar V_\xi^{''}(0)=&-\int_M\Big(\lagl\triangle_{M}^\bot(\eta) ,\eta\ragl-\lagl D^\bot_{x^\top+A_\xi(x^\top)}\eta
+g^{ik}g^{jl}\lagl h_{ij},\eta\ragl h_{kl} +\eta,\eta\ragl\nnm\\
&+\lagl \nabla \lagl x,\xi\ragl, D_{\pp{}{t}}\pp{F}{t}|_{t=0}\ragl\Big)e^{-\bar f}dV.\label{2wf}
\end{align}
\end{thm}

In order to simplify the second variation formulas we introduce the following definition:

\begin{dfn}[$SN$-variation]\label{uni-nml-var}\rm A variation $F:M^m\times (-\veps,\veps)\to \bbr^{m+p}$ of an immersion $x:M^m\to \bbr^{m+p}$ is called specially normal (or simply $SN$) if it is normal and $\pp{^2F}{t^2}|_{t=0}=0$.
\end{dfn}

Clearly, for any $\eta\in\Gamma(T^\bot M)$, $SN$-variations with variation vector field $\eta$ do exist. For example, we can choose
$$F(p,t)=x(p)+\psi(t)\eta(p),\quad \forall\, (p,t)\in M^m\times(-\veps,\veps)$$
where $\psi$ is any smooth function satisfying $\psi(0)=\psi''(0)=0$, $\psi'(0)=1$.

\begin{cor}[The simplified second variation formulas]\label{smplf} Let $x:M^m\to\bbr^{m+p}$ be a $\xi$-submanifold. Then for any compact $SN$-variation $F:M^m\times(-\veps,\veps)\to\bbr^{m+p}$ it holds that
\begin{align}
V_{\xi}^{''}(0)=&-\int_M\Big(\lagl(\triangle_{M}^\bot-D^\bot_{x^\top+A_\xi(x^\top)}+1)\eta+g^{ik}g^{jl}\lagl h_{ij},\eta\ragl h_{kl} ,\eta\ragl\Big)e^{-f}dV,\label{s2vf}\\
\bar V_\xi^{''}(0)=&-\int_M\Big(\lagl(\triangle_{M}^\bot-D^\bot_{x^\top+A_\xi(x^\top)}+1)\eta+g^{ik}g^{jl}\lagl h_{ij},\eta\ragl h_{kl} ,\eta\ragl\Big)e^{-\bar f}dV.\label{s2wf}
\end{align}
\end{cor}

{\rmk\label{rmk}\rm From the above discussion, one may naturally think of the variational characterization of the usual submanifolds with parallel mean curvature vector in the Euclidean space. In fact, our computations and argument of those two sections essentially apply to this situation. In particular, a suitable functional $\td V_\xi$ may be defined by}
$$
\td V_\xi=\int_M e^{\lagl x,\xi\ragl}dV_x,\quad \forall x\in {\mathcal M}
$$
and the first variation formula of $\td V_\xi$ is given in the following
\begin{prop}\label{prop2} Let $x\in {\mathcal M}$ be fixed and $\xi:M^m\to\bbr^{m+p}$ be a smooth map. Suppose that $F$ is a compact variation of $x$. Then
\be\label{1tdv}
\td V_\xi'(t)=-\int_M
\lagl(H_{t}-\xi)+\nabla^t\lagl x_t,\xi\ragl,\pp{F}{t}\ragl e^{\lagl x,\xi\ragl}dV_t.
\ee
In particular, if $F$ is a normal variation of $x$, then
\be\label{tdv0}
\td V_\xi'(0)=-\int_M\lagl H-\xi,\eta\ragl e^{\lagl x,\xi\ragl}dV.
\ee
\end{prop}

\begin{cor}\label{cor4}
An immersion $x\in{\mathcal M}$ has a parallel mean curvature vector if and only if there exists a parallel normal vector field $\xi\in\Gamma(T^\bot M)$ such that $x$ is the critical point of the functional $\td V_\xi$ for all the compact normal variations of $x$.
\end{cor}

Moreover, the second variation formula for a submanifold $x:M^m\to \bbr^{m+p}$ with parallel mean curvature vector $H\equiv\xi$ should be described as

\begin{thm}\label{main5} Let $x:M^m\to\bbr^{m+p}$ be an immersed submanifold with parallel mean curvature $H$. Then for any compact normal variation $F:M^m\times(-\veps,\veps)\to\bbr^{m+p}$ we have
\be
\td V_H^{''}(0)=-\int_M\Big(\lagl\triangle_{M}^\bot(\eta) +D^\bot_{\nabla\lagl x,H\ragl} \eta,\eta\ragl+|A_\eta|^2+
\lagl \nabla \lagl x,H\ragl, D_{\pp{}{t}}\pp{F}{t}|_{t=0}\ragl\Big)e^{\lagl x,H\ragl}dV.\label{2tdv}
\ee
\end{thm}

\section{The instabilities of the canonical examples}

The most natural stability definition to the functional $V_\xi$ is as follows:

\begin{dfn}\label{stblty} A $\xi$-submanifold $x:M^m\to \bbr^{m+p}$ is called stable if $V_\xi(x)<+\infty$ and for every $SN$-variation $F:M^m\times (-\veps,\veps)\to \bbr^{m+p}$ of $x$ it holds that $V_\xi''(0)\geq 0$ or, equivalently, $\bar V_\xi''(0)\geq 0$.
\end{dfn}

In this section we shall show that,  as $\xi$-submanifolds, all the canonical examples given in Section 2 are not stable in the sense of Definition \ref{stblty}.

Write the second fundamental form $h$ of $x$ locally as $h=h_{ij}\omega^i\omega^j=h^\alpha_{ij}e_\alpha$ with respect to an orthonormal tangent frame field $\{e_i;\ 1\leq i\leq m\}$ with dual $\{\omega^i\}$ and an orthonormal normal frame field $\{e_\alpha;\ m+1\leq\alpha\leq m+p\}$, and denote
\be\label{5.1} {\mathcal L}=\triangle^\bot_{M^m}-D^\bot_{x^\top+A_\xi(x^\top)},\quad L={\mathcal L}+\lagl h_{ij},\cdot\ragl h_{ij} +1,\quad \td{\mathcal L}=\triangle_{M^m} -\nabla_{x^\top+A_\xi(x^\top)},\ee
where $\triangle^\bot_{M^m}$, $\triangle_{M^m}$ are Laplacians on $T^\bot M^m$, $TM^m$ respectively, and sometimes we shall omit the subscript ``$_{M^m}$'' if no confusion is made. It follows that
\be
Q(\eta,\eta):\equiv V_{\xi}^{''}(0)= -\int_M\lagl L(\eta) ,\eta\ragl e^{-f}dV,
\ee
and that, for any parallel normal vector field $N$,
\be\label{N}L(N)=N+\lagl h_{ij},N\ragl h_{ij}.\ee

\begin{lem}\label{lem6.2}
\be
L(\phi\eta)=(\td{\mathcal L}\phi)\eta+\phi L(\eta)+2D^\bot_{\nabla \phi}\eta,\quad \phi\in C^\infty(M^m),\, \eta\in\Gamma(T^\bot M^m).\label{eq6.5}
\ee
\end{lem}

\begin{proof} We compute directly
\begin{align*}
L(\phi\eta)=&\triangle^\bot(\phi\eta)-D^\bot_{x^\top+A_\xi(x^\top)}(\phi\eta)+\lagl h_{ij},\phi\eta\ragl h_{ij}+\phi\eta\\
=&(\triangle\phi)\eta+2D^\bot_{\nabla\phi}\eta+\phi\triangle^\bot\eta-(\nabla_{x^\top+A_\xi(x^\top)}\phi)\eta\\
&-\phi(D^\bot_{x^\top+A_\xi(x^\top)}\eta)+\phi\lagl h_{ij},\eta\ragl h_{ij}+\phi\eta\\
=&(\triangle-\nabla_{x^\top+A_\xi(x^\top)})\phi\eta+\phi(\triangle^\bot
-D^\bot_{x^\top+A_\xi(x^\top)}+\lagl h_{ij},\cdot\ragl h_{ij}+1)\eta+2D^\bot_{\nabla\phi}\eta\\
=&(\td{\mathcal L}\phi)\eta+\phi(L\eta)+2D^\bot_{\nabla\phi}\eta.
\end{align*}
\end{proof}

\begin{lem} Let $x:M^m\to \bbr^{m+p}$ be a $\xi$-submanifold. Then for any $\eta_1,\eta_2\in\Gamma(T^\bot M^m)$ one of which is compactly supported, it holds that
\be
\int_M\lagl\eta_1,{\mathcal L}\eta_2\ragl e^{-f}dV=-\int_M\lagl D^\bot\eta_1,D^\bot\eta_2\ragl e^{-f}dV.\label{eq6.7}
\ee
Similarly,  for any  $\phi_1, \phi_2\in C^\infty(M^m)$ one of which is compactly supported, it holds that
\be
\int_M\phi_1\td{\mathcal L}\phi_2 e^{-f}dV=-\int_M\lagl\nabla \phi_1,\nabla\phi_2\ragl e^{-f}dV.\label{eq6.7'}
\ee
\end{lem}

\begin{proof} To prove the two formulas, it suffices to use the Divergence Theorem and the following equalities:
\begin{align}
&\lagl\eta_1,{\mathcal L}\eta_2\ragl e^{-f}=\dv\left(\lagl\eta_1,D^\bot_{e_i}\eta_2\ragl e^{-f}e_i\right)-\lagl D^\bot\eta_1,D^\bot\eta_2\ragl e^{-f}, \label{eq6.9}\\
&\phi_1\td{\mathcal L}\phi_2e^{-f}=\dv\left(\phi_1\nabla_{e_i}\phi_2e^{-f}e_i\right) -\lagl\nabla\phi_1,\nabla\phi_2\ragl e^{-f}.\label{eq6.9'}
\end{align}
\end{proof}

\begin{lem}\label{lem6.4}  For any $\phi\in C^\infty_0(M^m)$ and $\eta\in\Gamma(T^\bot M^m)$, it holds that
\be\label{phieta}
\int_M \lagl \phi\eta,L(\phi\eta)\ragl e^{-f}dV=\int_M\phi^2 \lagl \eta,L(\eta)\ragl e^{-f}dV-\int_M|\nabla \phi|^2 |\eta|^2e^{-f}dV.
\ee
\end{lem}

\begin{proof} By \eqref{eq6.5} and \eqref{eq6.7'}, we find
\begin{align*}
\int_M \lagl \phi\eta,&L(\phi\eta)\ragl e^{-f}dV=\int_M  \lagl \phi\eta,(\td{\mathcal L}\phi)\eta+\phi L\eta+2D^\bot_{\nabla\phi}\eta\ragl e^{-f}dV\\
=&\int_M (\phi|\eta|^2)\td{\mathcal L}\phi e^{-f}dV+\int_M\phi^2\lagl\eta,L\eta\ragl e^{-f}dV+\int_M\lagl\eta,D^\bot_{\nabla\phi^2}\eta\ragl e^{-f}dV\\
=&-\int_M((|\nabla\phi|^2|\eta|^2)+\fr12\lagl\nabla\phi^2,\nabla|\eta|^2\ragl) e^{-f}dV+\int_M\phi^2\lagl\eta,L\eta\ragl e^{-f}dV\\
&+\fr12\int_M\nabla_{\nabla\phi^2}|\eta|^2 e^{-f}dV\\
=&\int_M\phi^2\lagl\eta,L\eta\ragl e^{-f}dV-\int_M  |\nabla\phi|^2|\eta|^2 e^{-f}dV.
\end{align*}
\end{proof}

\begin{prop}\label{prop mp} As $\xi$-submanifolds, all $m$-planes in $\bbr^{m+p}$ are not stable.
\end{prop}

\begin{proof}  For an $m$-plane $x:P^m\subset \bbr^{m+p}$, let $o$ be the orthogonal projection on $P^m$ of the origin $O$. Then $\xi=\strl{\ra}{Oo}$.
Denote by $B_R(o)\subset P$ the closed ball of radius $R>0$ centered at the fixed point $o$:
$$
B_R(o)=\{x\in P;\ |x^\top|\equiv |x-\xi|\leq R\}.
$$
Let $N$ be a unit constant vector in $\bbr^m$ orthogonal to $P^m$ and $\phi_R$ be a cut-off function on $P^m$ satisfying
$$
(\phi_R)|_{B_R(o)}\equiv 1,\quad (\phi_R)|_{P^m\bsl B_{R+2}(o)}\equiv 0,\quad |\nabla\phi|\leq 1,\quad R>0.
$$
Define $\eta_R=\phi_RN$. Then $\eta_R$ is compactly supported and can be chosen a variation vector field for some $SN$-variation. By \eqref{phieta} and \eqref{N},
\begin{align*}
Q(\eta_R,\eta_R)=&-\int_M \lagl \phi_R N,L(\phi_R N)\ragl e^{-f}dV \\ =&-\int_{P^m}\phi_R^2 \lagl N,L(N)\ragl e^{-f}dV+\int_{P^m}|\nabla \phi_R|^2 e^{-f}dV\\
=&-\int_{P^m}\phi_R^2 \lagl N,N+\lagl h_{ij},N\ragl h_{ij}\ragl e^{-f}dV+\int_{P^m}|\nabla \phi_R|^2 e^{-f}dV\\
\leq& -\int_{P^m}\phi_R^2 e^{-f}dV+\int_{B_{R+2}(o)\bsl B_R(o)} e^{-f}dV
\ \to -\int_{P^m} e^{-f}dV<0
\end{align*}
when $R\,\to +\infty$ since $\int_{P^m} e^{-f}dV<+\infty$. Thus for large $R$ we have $Q(\eta_R,\eta_R)<0$.
\end{proof}
\begin{prop}\label{prop1} As $\xi$-submanifolds, the standard $m$-spheres $S^m(r)$ are all non-stable.
\end{prop}

\begin{proof} For the standard sphere $S^m(r)\subset\bbr^{m+1}\subset\bbr^{m+p}$, we have $h=-\fr1{r^2}g\,x$, $x^\bot=x$ and $\xi=\left(-\fr m{r^2}+1\right)x$. Choose the variation vector field $\eta=x$ so that ${\mathcal L}\eta=0$. It follows that
\begin{align*}
Q(\eta,\eta)\leq &-\int_{S^m(r)}\lagl\eta,L(\eta)\ragl e^{-f}dV_{S^m(r)}=-\int_{S^m(r)}(\sum\lagl h_{ij},\eta\ragl^2+|x|^2) e^{-f}dV_{S^m(r)}\\
=&-(m+r^2)\int_{S^m(r)}e^{-f}dV_{S^m(r)}<0.
\end{align*}
\end{proof}

From Proposition \ref{prop mp} and Proposition \ref{prop1}, we easily find

\begin{cor}\label{cor2} The product $\xi$-submanifolds $S^{m_1}(r_1)\times\cdots\times S^{m_k}(r_k)\times P^{n_1}\times\cdots\times P^{n_l}$ are not stable.\end{cor}

A more general conclusion than Proposition \ref{prop1} is the following

\begin{prop}\label{prop5.4} Let $x:M^m\to\bbr^{m+p}$ be a compact $\xi$-submanifold. If $x$ has a non-trivial parallel normal vector field, then $x$ is not stable. In particular, all compact $\lambda$-hypersurfaces and compact $\xi$-submanifold with $\xi\neq 0$ are not stable.
\end{prop}

\begin{proof}
Let $\eta\neq 0$ be a parallel normal vector field. Then $\eta$ can be chosen to be a variation vector field of some $SN$-variation $F$ of $x$. Since $\triangle^\bot\eta=D^\bot_{x^\top+A_\xi(x^\top)}\eta=0$, it then follows from \eqref{s2vf} that
$$
Q(\eta,\eta)=-\int_M(\sum\lagl h_{ij},\eta\ragl^2+|\eta|^2)e^{-f}dV<0.
$$
\end{proof}

\begin{cor}\label{cor3}
Any compact and simply connected $\xi$-submanifold with flat normal bundle is not stable.
\end{cor}

\section{The $W$-stability of $\xi$-submanifolds}

By the discussion of last section, it turns out that the stability given in Definition \ref{stblty} is over-strong in a sense. So it is natural and interesting to find some weaker stability for $\xi$-submanifolds. Motivated by the ``volume-preserving'' variations in the case of hypersurfaces (see \cite{mc-r}), we introduce the $W$-stability in the following way.

\begin{dfn} \rm Let $x:M^m\to \bbr^{m+p}$ be an immersion. A $SN$-variation $F:M^m\times (-\veps,\veps)\to \bbr^{m+p}$ of $x$ is called $VP$ if the corresponding variation vector $\eta\equiv\pp{F}{t}|_{t=0}$ satisfies
\be\label{vp}\int_M\lagl\eta,N\ragl e^{-f}=0,\quad\forall\, N\in\Gamma(T^\bot M)\text{ and } D^\bot N\equiv 0.\ee
\end{dfn}

\begin{rmk}\rm It is clear that, in the special case of codimension $1$, a $VP$-variation is nothing but the ``volume-preserving'' one which has been considered in \cite{mc-r}.\end{rmk}

\begin{dfn}\label{wstblty}\rm A $\xi$-submanifold $x:M^m\to \bbr^{m+p}$ is called $W$-stable if $V_\xi(x)<+\infty$ and for every $VP$-variation it holds that $V_\xi''(0)\geq 0$.
\end{dfn}

Then we have

\begin{thm}\label{main4} The $m$-planes are all $W$-stable.\end{thm}

\begin{proof} For an $m$-plane $x:P^m\subset \bbr^{m+p}$, let $\eta$ be an arbitrary normal vector field on $P^m$ with compact support. Then we have $A_\eta\equiv 0$, $x-\xi=x^\top$ and
$$L=\triangle_{P^m}^\bot -D^\bot_{x^\top}+1.$$
Clearly, there are constant normal basis $e_\alpha$, $\alpha=m+1,\cdots,m+p$. So $\eta$ can be expressed by
$\eta=\sum\eta^\alpha e_\alpha$ with $\eta^\alpha\in C^\infty_0(P^m)$. Consequently,
$$
L(\eta)=\sum \td L(\eta^\alpha)e_\alpha,\quad \lagl L\eta,\eta\ragl=\sum \eta^\alpha \td L\eta^\alpha,
$$
where $\td L=\triangle_{P^m}-\nabla_{x^\top}+1$.  Now we make the following

{\bf Claim:}  {\em the eigenvalues of the operator $-\td L$ are $n-1$ with $n=0,1,\cdots$.}

To prove this claim, we need to make use of the {\em multi-variable Hermitian polynomials} ${\mathcal H}_{n_1\cdots n_m}$ on $\bbr^m$, labelled with $0\leq n_1,\cdots, n_m<+\infty$, which are defined by the expansion (see \cite{dat1} and \cite{dat} for the detail)
\begin{align}\label{herm1}
e^{-\fr{|u-t|^2}{2}}=&e^{-\fr{|u|^2}{2}}\sum_{n_1,\cdots, n_m}\fr{(t^1)^{n_1}\cdots (t^m)^{n_m}}{n_1!\cdots n_m!}{\mathcal H}_{n_1\cdots n_m}(u),
\nnm\\
&\quad\qquad u=(u^1,\cdots,u^m),\ t=(t^1,\cdots t^m)\in\bbr^m,
\end{align}
or equivalently
\begin{align}\label{herm2}
e^{-\fr{|t|^2}{2}+\lagl t,u\ragl}=&\sum_{n_1,\cdots, n_m}\fr{(t^1)^{n_1}\cdots (t^m)^{n_m}}{n_1!\cdots n_m!}{\mathcal H}_{n_1\cdots n_m}(u),\nnm\\
&\quad u=(u^1,\cdots,u^m),\ t=(t^1,\cdots t^m)\in\bbr^m,
\end{align}
It is clear that
\be\label{hermto}
{\mathcal H}_{n_1\cdots n_m}(u)={\mathcal H}_{n_1}(u^1)\,\cdots\,{\mathcal H}_{n_m}(u^m),\quad\forall u=(u^1,\cdots,u^m)\in\bbr^m
\ee
where, for each $i=1,\cdots,m$, ${\mathcal H}_{n_i}(u^i)$ is the Hermitian Polynomial of one variable $u^i$ defined by
\be\label{herone}
e^{-\fr12|t^i|^2+u^it^i}=\sum_{n_i}\fr{(t^i)^{n_i}}{n_i!}{\mathcal H}_{n_i}(u^i),\quad u^i,t^i\in\bbr.
\ee
By \eqref{herone}, we easily find that
\be\label{herone1}
{\mathcal H}_{n_i+1}(u^i)=u^i{\mathcal H}_{n_i}-n_i{\mathcal H}_{n_i-1},\quad
\dd{}{u^i}{\mathcal H}_{n_i}(u^i)=n_i{\mathcal H}_{n_i-1},\ i=1,\cdots,m
\ee
implying that
\be\label{herone2}
\left(-\dd{^2}{(u^i)^2}+u^i\dd{}{u^i}\right){\mathcal H}_{n_i}(u^i)=n_i{\mathcal H}_{n_i}(u^i),\quad i=1,\cdots,m.
\ee
Consequently, by \eqref{hermto}, we have
\be\label{hermlap}
\left(-\triangle_{\bbr^m}+\nabla_{u}\right){\mathcal H}_{n_1\cdots n_m}(u)=\big(\sum_{i=1}^mn_i\big){\mathcal H}_{n_1\cdots n_m}(u),\quad\forall n_1,\cdots,n_m\geq 0.
\ee
It is known that all these multi-variable Hermitian polynomials are weighted square integrable with the weight $e^{-\fr{|u|^2}{2}}$, that is
$$
{\mathcal H}_{n_1\cdots n_m}\in L^2_w(\bbr^m):=\ol{\{\vfi\in C^\infty(\bbr^m);\ \int_{\bbr^m}\vfi^2 e^{-f}dV_{\bbr^m}<+\infty\}}.
$$
Consequently, integers $\sum_{i=1}^m n_i=0,1,\cdots$ are eigenvalues of the operator $-\triangle_{\bbr^m}+\nabla_{u} $ acting on $L^2_w(\bbr^m)$. By making a change of coordinates on $\bbr^{m+p}$ we can assume $x^i-\xi^i=u^i$, $i=1,2,\cdots, m$, for $x\in P^m$. Thus \eqref{hermlap} shows that $-\td L+1$ has $n=0,1,\cdots$ as its eigenvalues, or equivalently, $n-1=-1,0,1,\cdots$ are eigenvalues of $-\td L$ where constants are those eigenfunctions corresponding to $-1$.

To complete the claim, we also have to show that $\{{\mathcal H}_{n_1\cdots n_m};\ n_1,\cdots,n_m\geq 0\}$ is a complete basis for the space $S^{\infty,2}_w(\bbr^m)$ of smooth and weighted square integrable functions on $\bbr^m$. For doing this, we let $E$ be the orthogonal complement in $L^2_w(\bbr^m)$ of the closure of the linear span of all ${\mathcal H}_{n_1\cdots n_m}$, that is,
$$
E:=(\ol{\spn\{{\mathcal H}_{n_1\cdots n_m},\ n_1,\cdots, n_m=0,1,\cdots\}})^\bot.
$$
For any $\vfi\in E$, we have
$$
0=(\vfi,{\mathcal H}_{n_1\cdots n_m})_w:=\int_{\bbr^m}\vfi(u){\mathcal H}_{n_1\cdots n_m}(u)e^{-f}dV_{\bbr^m},\quad n_1,\cdots, n_m=0,1,\cdots.
$$
It then easily follows from \eqref{herm2} that ${\mathcal F}(\vfi e^{-f})=0$ where ${\mathcal F}$ is the usual multi-variable Fourier transformation. Since ${\mathcal F}$ is injective, we obtain that $\vfi e^{-f}=0$ implying $\vfi\equiv 0$. This shows that $E=0$ and thus \be\label{add}L^2_w(\bbr^m)=\ol{\spn\{{\mathcal H}_{n_1\cdots n_m},\ n_1,\cdots, n_m=0,1,\cdots\}}.\ee

Now suppose $\eta$ is a compact normal vector field that can be taken as a $VP$-variation vector field. Then for each $\alpha$, we have
$$\eta^\alpha\in S^{\infty,2}_w(P^m):=\{\vfi\in C^\infty(P^m);\ \int_{P^m}\vfi^2 e^{-f}dV_{P^m}<+\infty\}.$$ Since $\td L$ is self-adjoint with respect to the weighted measure $e^{-f}dV$, we know that it is diagonalizable, that is, any compactly supported smooth function can be decomposed into a sum of some eigenfunctions of $\td L$. In particular, we can write for each $\alpha=m+1,\cdots,m+p$,
\be\label{6.10}
\eta^\alpha=\eta^\alpha_0+\sum_{k\geq 1}\eta^\alpha_k,\quad \eta^\alpha_0\in\bbr,
\ee
where $\eta^\alpha_k\in S^{\infty,2}_{w}(P^m)$ satisfying $\td L(\eta^\alpha_k)=-\lambda_k\eta^\alpha_k$, $k\geq 0$. Furthermore, the self-adjointness of $\td L$ also implies that, for each pair of $k\neq l$, $\eta^\alpha_k$ and $\eta^\alpha_l$ are orthogonal, that is
\be\label{k,l}\int_{P^m}\sum_\alpha\eta^\alpha_k\eta^\alpha_le^{-f}dV=0,\quad k\neq l.\ee
Since $\eta$ is a $VP$-variation vector field, we have by \eqref{6.10} and \eqref{vp} that $\int_{P^m}\eta^\alpha e^{-f}dV=0$ for all $\alpha=m+1,\cdots,m+p$. It then follows from \eqref{k,l} that $\eta^\alpha_0=0$, $\alpha=m+1,\cdots,m+p$. Therefore
$$\int_{P^m}\sum_\alpha|\eta^\alpha|^2e^{-f}dV=\int_{P^m}\sum_\alpha\sum_{k,l\geq 1}\eta^\alpha_k\eta^\alpha_le^{-f}dV=\sum_\alpha\sum_{k\geq 1}\int_{P^m}|\eta^\alpha_k|^2e^{-f}dV.$$
Consequently, we have
\begin{align*}
\int_{P^m}\sum_\alpha\eta^\alpha(-\td L\eta^\alpha)e^{-f}dV=&\int_{P^m}\sum_\alpha\sum_{k\geq 1}\eta^\alpha_k\sum_{l\geq 1}(-\td L\eta^\alpha_l)e^{-f}dV\\
=&\sum_\alpha\sum_{k,l\geq 1}\int_{P^m}\lambda_l\eta^\alpha_k\eta^\alpha_le^{-f}dV
=\sum_\alpha\sum_{k\geq 1}\lambda_k\int_{P^m}|\eta^\alpha_k|^2e^{-f}dV\\
\geq&\lambda_1\sum_\alpha\sum_k\int_{P^m}|\eta^\alpha_k|^2e^{-f}dV =\lambda_1\sum_\alpha\int_{P^m}|\eta^\alpha|^2\geq 0
\end{align*}
implying that
\begin{align*}
Q(\eta,\eta)=&-\int_{P^m}\lagl\eta,L\eta\ragl e^{-f}dV=\int_{P^m}\sum_\alpha\eta^\alpha(-\td L\eta^\alpha)e^{-f}dV\\
=&\sum_\alpha\int_{P^m}\eta^\alpha(-\td L\eta^\alpha)e^{-f}dV\geq 0.
\end{align*}
\end{proof}

\begin{thm}\label{prop1'} As a $\xi$-submanifold, the index ${\rm ind}(S^m(r))$ of the standard $m$-sphere $S^m(r)$ with respect to $VP$-variations is no less than $m+1$. Furthermore, ${\rm ind}(S^m(r))=m+1$ if and only if $r^2\leq m$. In particular, all of these spheres are not $W$-stable.
\end{thm}

\begin{proof} For the standard sphere $S^m(r)\subset\bbr^{m+1}\subset\bbr^{m+p}$, we have $x^\top=0$, $h=-\fr1{r^2}gx$ and hence $\xi=\left(-\fr m{r^2}+1\right)x$. It follows that $x-\xi=\fr m{r^2}x$ and
$$
L=\triangle^\bot_{S^m(r)}+\lagl h_{ij},\cdot\ragl h_{ij}+1=\triangle^\bot_{S^m(r)}+\fr{m}{r^4}\lagl x,\cdot\ragl x+1,\quad \td{\mathcal L}=\triangle_{S^m(r)}.
$$
In particular, $L(x)=\fr1{r^2}(m+r^2)x$ and, for all parallel normal vector field $N$ orthogonal to $x$, $L(N)=N$. Let $e_{m+2},\cdots,e_{m+p}$ be an othonormal constant basis of the subspace $(\spn\{TS^m(r),x\})^\bot\subset\bbr^{m+p}$. Then
$e_{m+1}:\equiv\fr1r x,e_{m+2},\cdots,e_{m+p}$ is an othonormal normal frame field of $S^m(r)$ and
\be\label{le} L(e_{m+1})=\fr1{r^2}(m+r^2)e_{m+1},\quad L(e_\alpha)=e_\alpha,\quad \alpha=m+2,\cdots,m+p.\ee

Now for any $\eta\in\Gamma(T^\bot S^m(r))$ we can write
$$\eta=\sum_{\alpha}\eta^\alpha e_\alpha \text{ with } \eta^\alpha\in C^\infty(S^m(r)), \ m+1\leq\alpha\leq m+p.$$
Then by \eqref{eq6.5} and \eqref{le}
\begin{align*}
L(\eta)=&\sum_\alpha(\td{\mathcal L}(\eta^\alpha)) e_\alpha+\eta^\alpha L(e_\alpha) \\
=&((\triangle_{S^m(r)}\eta^{m+1}) e_{m+1}+\eta^{m+1} L(e_{m+1})+\sum_{\alpha\geq m+2}((\triangle_{S^m(r)}\eta^\alpha) e_\alpha+\eta^\alpha L(e_\alpha))\\
=&(\td L+\fr{m}{r^2})\eta^{m+1} e_{m+1} +\sum_{\alpha\geq m+2}\td L(\eta^\alpha) e_\alpha
\end{align*}
where $\td L=\triangle_{S^m(r)}+1$. Furthermore, let $\lambda_k$, $k\geq 0$ be the eigenvalues of $\td L$ and write
$\eta^\alpha=\sum_{k\geq 0}\eta^\alpha_k$ for some eigenfunctions $\eta^\alpha_k$ satisfying $\td L(\eta^\alpha_k)=-\lambda_k\eta^\alpha_k$, $k\geq 0$.

It is well-known that the eigenvalues of $-\triangle_{S^m(r)}$ is $\fr{k(m+k-1)}{r^2}$, $k\geq 0$, so that
$$\lambda_k=\fr{k(m+k-1)}{r^2}-1, \text{ for }k=0,1,2,\cdots,$$
with constants being the eigenfunctions corresponding to $k=0$. But by \eqref{vp}, $\int_{S^m(r)}\eta^\alpha e^{-f}dV_{S^m(r)}=0$ which implies that $\eta^\alpha_0=0$. Therefore,
\begin{align}
Q(\eta,\eta)=&-\int_{S^m(r)}\lagl\eta, L(\eta)\ragl e^{-f}dV_{S^m(r)}\nnm\\
=&-\int_{S^m(r)}\eta^{m+1}(\td L+\fr{m}{r^2})\eta^{m+1} e^{-f}dV_{S^m(r)}+\sum_{\alpha\geq m+2}\int_{S^m(r)}\eta^\alpha(-\td L\eta^\alpha) e^{-f}dV_{S^m(r)}\nnm\\
=&\sum_{k\geq 1}\int_{S^m(r)}\left(\fr{k(m+k-1)}{r^2}-\fr1{r^2}(m+r^2)\right)|\eta^{m+1}_k|^2 e^{-f}dV_{S^m(r)}\nnm\\
&+\sum_{\alpha\geq m+2,k\geq 1}\int_{S^m(r)}\left(\fr{k(m+k-1)}{r^2}-1\right)|\eta^\alpha_k|^2 e^{-f}dV_{S^m(r)}\nnm\\
=&-\int_{S^m(r)}|\eta^{m+1}_1|^2 e^{-f}dV_{S^m(r)}\nnm\\
&+\sum_{k\geq 2}\int_{S^m(r)}\left(\fr{k(m+k-1)}{r^2}-\fr1{r^2}(m+r^2)\right)|\eta^{m+1}_k|^2 e^{-f}dV_{S^m(r)}\nnm\\
&+\sum_{\alpha\geq m+2,k\geq 1}\int_{S^m(r)}\left(\fr{k(m+k-1)}{r^2}-1\right)|\eta^\alpha_k|^2 e^{-f}dV_{S^m(r)}\nnm\\
\geq&-\int_{S^m(r)}|\eta^{m+1}_1|^2 e^{-f}dV_{S^m(r)}+\left(\fr{m+2}{r^2}-1\right)\sum_{k\geq 2}\int_{S^m(r)}|\eta^{m+1}_k|^2 e^{-f}dV_{S^m(r)}\nnm\\
&+\left(\fr{m}{r^2}-1\right)\sum_{\alpha\geq m+2,k\geq 1}\int_{S^m(r)}|\eta^\alpha_k|^2 e^{-f}dV_{S^m(r)}.\label{eq6}
\end{align}
Define
$$V_{\lambda_1}=\{\vfi\in C^\infty(S^m(r));\ \triangle_{S^m(r))}\vfi=-\fr m{r^2}\vfi\},\quad
\td V_{\lambda_1}=\{\vfi e_{m+1};\ \vfi\in V_{\lambda_1}\}.
$$
Then $\dim\td V_{\lambda_1}=\dim V_{\lambda_1}$ and the left side is well-known to be $m+1$. It is not hard to see from \eqref{eq6} that $Q$ is negative definite on $\td V_{\lambda_1}$, and thus ${\rm ind}(S^m(r))\geq m+1$ with the equality holding if and only if $\fr{m}{r^2}-1\geq 0$, that is, $r^2\leq m$.
\end{proof}

\section{The uniqueness problem for complete $W$-stable $\xi$-submanifolds}

It is interesting to know wether or not $m$-planes are the only $W$-stable $\xi$-submanifolds. The following rigidity theorem (and one of its simple corollaries) can be taken as the first step in solving this problem:

\begin{thm}\label{main6} Let $x:M^m\to\bbr^{m+p}$ be a properly immersed, complete and $W$-stable $\xi$-submanifold with flat normal bundle. If the condition \eqref{A} is fulfilled, then $x(M^m)$ must be an $m$-plane.\end{thm}

\begin{cor}\label{cor7.1} Any properly immersed, complete and $W$-stable $\xi$-submanifold in $\bbr^{m+p}$ with flat normal bundle must be an $m$-plane if the Weigarten map $A_\xi$ vanishes identically.
\end{cor}

\begin{cor}\label{cor7.2} Any properly immersed, complete and $W$-stable self-shrinker in $\bbr^{m+p}$ with flat normal bundle must be an $m$-plane passing through the origin.
\end{cor}

The main motivation here is the idea used by \cite{mc-r} and we need to extend it to fit our consideration of higher codimension.

To prove Theorem \ref{main6}, we may first make use of the universal covering to assume that $M^m$ is simply connected. Then that $x$ has a flat normal bundle implies the existence of a parallel orthonormal normal frame $\{e_\alpha;\  m+1\leq \alpha\leq m+p\}$.

\begin{lem} Let $x$ be a $\xi$-submanifold. Then for any constant vector $v\in\bbr^{m+p}$ and any parallel normal vector field $N$, we have
\be\label{vn}
\td{\mathcal L}\lagl v,N\ragl=-\lagl A_N,A_{v^\bot}\ragl+\lagl A_N(v^\top),A_\xi(x^\top)\ragl,
\ee
where $v^\top$ and $v^\bot$ are the orthogonal projections of the vector $v$ on $TM^m$ and $T^\bot M^m$, respectively.
\end{lem}

\begin{proof} By using Weingarten formula and the equality that $D^\bot (H+x^\bot)\equiv 0$, we find
\begin{align*}\td{\mathcal L}\lagl v,N\ragl=&\triangle\lagl v,N\ragl-\nabla_{x^\top+A_\xi(x^\top)}\lagl v,N\ragl\\
=&(\lagl v,-A_N(e_i)\ragl)_{,i}-\lagl v,-A_N(x^\top+A_\xi(x^\top))\ragl\\
=&-\lagl h_{iji},N\ragl\lagl v,e_j\ragl-\lagl h_{ij},N\ragl\lagl v,e_j\ragl_{,i}+\lagl v,A_N(x^\top)\ragl+\lagl v,A_N(A_\xi(x^\top)\ragl\\
=&\lagl x,N\ragl_j\lagl v,e_j\ragl-\lagl h_{ij},N\ragl\lagl v,h_{ji}\ragl+\lagl v,A_N(x^\top)\ragl+\lagl v,A_N(A_\xi(x^\top)\ragl\\
=&-\lagl x,A_N(e_j)\ragl\lagl v,e_j\ragl-\lagl A_N,A_{v^\bot}\ragl+\lagl v,A_N(x^\top)\ragl+\lagl v,A_N(A_\xi(x^\top)\ragl\\
=&-\lagl x^\top,A_N(v^\top)\ragl-\lagl A_N,A_{v^\bot}\ragl+\lagl A_N(v^\top),x^\top\ragl+\lagl A_N (v^\top),A_\xi(x^\top)\ragl\\
=&-\lagl A_N,A_{v^\bot}\ragl+\lagl A_N (v^\top),A_\xi(x^\top)\ragl.
\end{align*}
\end{proof}

From \eqref{N}, \eqref{vn} and \eqref{eq6.5} we can easily find

\begin{lem}\label{lem6.2'} For a $\xi$-submanifold $x$, it holds that
\be L(v^\bot)= v^\bot+h(A_\xi(x^\top),v^\top)=v^\bot+\lagl h_{ik},\xi\ragl\lagl x,e_i\ragl\lagl v,e_j\ragl h_{kj},\quad\forall\, v\in\bbr^{m+p}.\label{lvbot}\ee
\end{lem}

In what follows, we always assume that the condition \eqref{A} is fulfilled. In this case, \eqref{vn} and \eqref{lvbot} reduce respectively to
\be\label{vn-vbot}\td{\mathcal L}\lagl v,N\ragl=-\lagl h_{ij},N\ragl\lagl h_{ij},v\ragl,\quad L(v^\bot)=v^\bot,\quad\forall\,v\in\bbr^{m+p}.\ee

\begin{lem}\label{lem6.3} For any $\eta=e_\alpha+v^\bot$, $v\in\bbr^{m+p}$, it holds that
\be
Q(\phi\eta,\phi\eta)\leq -\int_M \phi^2|\eta|^2e^{-f}dV+\int_M |\nabla\phi|^2(|\eta|^2+|v^\top|^2)e^{-f}dV,\quad\forall\,\phi\in C^\infty_0(M^m).
\ee
\end{lem}

\begin{proof} By \eqref{N} and \eqref{vn-vbot}, $$L(\eta)=L(e_\alpha+v^\bot)=e_\alpha+h^\alpha_{ij}h_{ij}+v^\bot=\eta+h^\alpha_{ij}h_{ij}.$$
It follows from \eqref{phieta} that
\begin{align}
&Q(\phi\eta,\phi\eta)=-\int_M \lagl \phi \eta,L(\phi \eta)\ragl e^{-f}dV\nnm\\
=&-\int_M\phi^2\lagl\eta,L(\eta)\ragl e^{-f}dV+\int_M |\nabla\phi|^2|\eta|^2 e^{-f}dV\nnm\\
=&-\int_M \phi^2\lagl\eta,\eta+h^\alpha_{ij}h_{ij}\ragl e^{-f}dV+\int_M |\nabla\phi|^2|\eta|^2 e^{-f}dV\nnm\\
=&-\int_M \phi^2|\eta|^2 e^{-f}dV
-\int_M \phi^2 h^\alpha_{ij}\lagl h_{ij},e_\alpha+v^\bot\ragl e^{-f}dV+\int_M |\nabla\phi|^2|\eta|^2 e^{-f}dV.\label{qphieta}
\end{align}

On the other hand, by \eqref{5.1} and \eqref{eq6.7}
\begin{align*}
&\int_M\phi^2\lagl e_\alpha,v^\bot\ragl e^{-f}dV=\int_M\phi^2\lagl e_\alpha, L(v^\bot)\ragl e^{-f}dV\\
=&\int_M \phi^2\lagl e_\alpha,\lagl h_{ij},v^\bot\ragl h_{ij}+v^\bot\ragl e^{-f}dV
+\int_M\lagl\phi^2 e_\alpha,{\mathcal L}v^\bot\ragl e^{-f}dV\\
=&\int_M\phi^2\lagl e_\alpha,v^\bot\ragl e^{-f}dV+\int_M\phi^2h^\alpha_{ij}\lagl h_{ij}, v^\bot\ragl e^{-f}dV-\int_M \lagl D^\bot(\phi^2e_\alpha),D^\bot v^\bot\ragl e^{-f}dV\\
=&\int_M\phi^2\lagl e_\alpha,v^\bot\ragl e^{-f}dV+\int_M\phi^2h^\alpha_{ij}\lagl h_{ij}, v^\bot\ragl e^{-f}dV-2\int_M \phi\lagl\nabla (\phi)e_\alpha,-d(v^\top)\ragl e^{-f}dV\\
=&\int_M\phi^2\lagl e_\alpha,v^\bot\ragl e^{-f}dV+\int_M\phi^2h^\alpha_{ij}\lagl h_{ij}, v^\bot\ragl e^{-f}dV+2\int_M \phi h^\alpha(\nabla \phi,v^\top)e^{-f}dV,\end{align*}
implying that
\begin{align*}
&\left|\int_M\phi^2h^\alpha_{ij}\lagl h_{ij}, v^\bot\ragl e^{-f}dV\right|=\left|2\int_M \phi h^\alpha(\nabla \phi,v^\top)e^{-f}dV\right|\\
\leq& 2\int_M |\phi||h^\alpha||\nabla\phi||v^\top|e^{-f}dV\leq \int_M \phi^2|h^\alpha|^2e^{-f}dV+\int_M|\nabla\phi|^2 |v^\top|^2e^{-f}dV.
\end{align*}
Inserting this into \eqref{qphieta} we complete the proof.
\end{proof}

Define
\be\label{WV}W=\spn_\bbr\{e_\alpha\},\quad V^\top=\{v^\top; v\in\bbr^{m+p}\}\quad V^\bot=\{v^\bot; v\in\bbr^{m+p}\}.\ee
Then $W$ is the space of parallel normal fields of $x$ and $p\leq\dim V^\bot\leq m+p$.

\begin{lem} Denote
\be\label{v0v1}V^\bot_0=\{v^\bot=\const;\ v\in\bbr^{m+p}\}\ee
Then $W\cap V^\bot=V^\bot_0$.
\end{lem}

\begin{proof}
For any $\eta\in W\cap V^\bot$, we have $\eta=v^\bot=c^\alpha e_\alpha$ for some $v\in\bbr^{m+p}$ and $c^\alpha\in\bbr$. Then it follows from \eqref{N} and \eqref{eq6.5} that
$$
v^\bot=L(v^\bot)=c^\alpha L(e_\alpha)=c^\alpha (e_\alpha+h^\alpha_{ij}h_{ij})=v^\bot +c^\alpha h^\alpha_{ij}h_{ij}
$$
implying that $c^\alpha h^\alpha_{ij}h_{ij}=0$. Multiplying this with $v^\bot=c^\alpha e_\alpha$ it follows that
$$\lagl h,v^\bot\ragl^2=\sum_{i,j,\alpha,\beta}c^\alpha c^\beta h^\alpha_{ij}h^\beta_{ij}=0.$$
Thus $\lagl h,v^\bot\ragl=0$ or equivalently $A_{v^\bot}=0$ which with the fact that $v^\bot$ is parallel in the normal bundle shows that $v^\bot$ must be a constant vector.

The inverse part is trivial.
\end{proof}

Define
$$\Gamma^{\infty,2}_w(T^\bot M^m):=\{\eta\in\Gamma(T^\bot M);\ \int_M|\eta|^2 e^{-f}dV<+\infty\},$$
on which there is a standard $L^2_w$-inner product $(\cdot,\cdot)$ by
$$(\eta_1,\eta_2):=\int_M\lagl\eta_1,\eta_2\ragl e^{-f}dV,\quad\forall\,\eta_1,\eta_2\in \Gamma^{\infty,2}_w(T^\bot M^m),
$$
giving the corresponding $L^2_w$-norm $\|\cdot\|_{2,w}$. The $L^2_w$-product $(\cdot,\cdot)$ and $L^2_w$-inner norm $\|\cdot\|_{2,w}$ for all weighted square integrable functions on $M^m$ are defined similarly.

Let $V^\bot_1$ be the orthogonal complement of $V^\bot_0$ in $V^\bot$ with respect to the $L^2_w$-inner product, and define $V=W\oplus V^\bot_1$ as subspaces of $\Gamma^{\infty,2}_w(T^\bot M^m)$. Since
$\dim W=p$ and $\dim V^\bot_1\leq \dim V^\bot\leq m+p$, $V$ is finite dimensional which implies that the standard sphere ${\mathbb S}=\{\eta\in V;\ \|\eta\|_{2,w}=1\}\subset V$ is compact.

Now we consider the compact case and prove the following

\begin{prop}\label{cmpt} Any compact $\xi$-submanifold, satisfying condition \eqref{A}, with parallel normal bundle can not be $W$-stable.\end{prop}

\begin{proof} It suffices to show that both of the following two are true:

(1) $Q$ is negative definite on $V$ and, consequently, is negative definite on $V^\bot_1$;

(2) $\dim V^\bot_1>0$.

In fact, the conclusion (1) follows directly from Lemma \ref{lem6.3} by choosing $\phi\equiv 1$; while conclusion (2) follows from the fact that the converse of (2) would imply that $M^m=\bbr^m$, by the argument at the end of this paper, which contradicts the compactness assumption.
\end{proof}

Next we consider the non-compact case and thus assume that $x:M^m\to\bbr^{m+p}$ is a complete and non-compact $\xi$-submanifold.

Let $o$ be a fixed point of $M$ and $\bar o=x(o)$. For any $R>0$ we define $\bar B_R(\bar o)=\{x\in\bbr^{m+p};\ |x-\bar o|\leq R\}$ and introduce a cut-off function $\bar\phi_R$ as follows (cf. \cite{mc-r}):
\be
\bar\phi_R(x)=\begin{cases} 1,& x\in \bar B_R(\bar o);\\
1-\fr1R(|x-\bar o|-R),& x\in \bar B_{2R}(\bar o)\bsl \bar B_R(\bar o);\\
0,&x\in\bbr^{m+p}\bsl\bar B_{2R}(\bar o).\end{cases}
\ee
For the given immersion $x:M^m\to\bbr^{m+p}$, let $\phi_R=\bar\phi_R\circ x\in C^\infty(M^m)$ and $B_R(o)=x^{-1}(\bar B_R(\bar o))$. Then $B_R(o)$ is compact since $x$ is properly immersed. In particular, $\phi_R$ is compactly supported. Furthermore, it is easily seen that
$|\nabla\phi_R|\leq |D\bar\phi_R|\leq\fr1R$.

\begin{lem}\label{R0} There is a large $R_0>0$ such that
$$
\int_{B_R(o)}|\eta|^2e^{-f}dV\geq \int_{B_{R_0}(o)}|\eta|^2e^{-f}dV>0,\quad \forall\,\eta\in {\mathbb S},\quad\forall\,R\geq R_0.
$$
\end{lem}

\begin{proof}
If the lemma is not true, then one can find a sequence $\{\eta_j\}\subset {\mathbb S}$ such that
$$\int_{B_j(o)}|\eta_j|^2e^{-f}dV=0, \quad j=1,2,\cdots.$$
By the compactness of $\mathbb S$, there exists a subsequence $\{\eta_{j_k}\}$ which is convergent to some $\eta_0\in {\mathbb S}$. For any $R>0$, there exists some $K>0$ such that $j_k>R$ for all $k>K$. It follows that
$$
\int_{B_R(o)}|\eta_0|^2e^{-f}dV=\lim_{k\to+\infty}\int_{B_R(o)}|\eta_{j_k}|^2e^{-f}dV =0
$$
which implies that
$$
\int_M|\eta_0|^2e^{-f}dV=\lim_{R\to+\infty}\int_{B_R(o)}|\eta_0|^2e^{-f}dV =0.
$$
Thus we have $\eta_0=0$ contradicting to the fact that $\eta_0\in {\mathbb S}$.
\end{proof}

For each $R>0$, define
\be\label{7.9-1}
m_R:=\min_{\eta\in{\mathbb S}}\{\int_M\phi_R^2|\eta|^2e^{-f}dV\},\quad M_R=\max_{\eta\in{\mathbb S}}\{\int_M\phi_R^2|\eta|^2e^{-f}dV\}.
\ee
Clearly,
\be\label{7.9-2}M_R\leq C:\equiv\max_{\eta\in{\mathbb S}}\int_M|\eta|^2e^{-f}dV<+\infty.\ee
Moreover, $m_R$ is increasing with respect to $R$ which together with Lemma \ref{R0} gives that
\be\label{7.9-3}m_R\geq m_{R_0}>0,\quad \forall\,R\geq R_0.\ee

\begin{lem}
There exists a large $R_0$, such that
\be\label{7.11}
\dim \phi_RV =\dim V,\quad \dim \phi_RV^\bot_1=\dim V^\bot_1,\quad R\geq R_0;
\ee
Furthermore, $Q$ is negative definite on $\phi_RV \supset\phi_RV^\bot_1$.
\end{lem}
\begin{proof}
First, we prove $\dim \phi_RV=\dim V$ for all $R\geq R_0$ if $R_0$ is large enough. For a given $R>0$, consider the surjective linear map
$$\Phi_R:V\to\phi_R V,\quad \eta\mapsto \Phi_R(\eta):=\phi_R\eta,\quad \forall\, \eta\in V.$$

We claim that, when $R_0$ is large enough,  the kernel $\ker\Phi_{R_0}$ of $\Phi_{R_0}$ must be trivial. In fact, if it is not the case, there should be a nonzero sequence $\{\eta_j\in V\}$ such that $\phi_j\eta_j=0$.
Define
$\td\eta_j=\fr{\eta_j}{\|\eta_j\|_{2,w}}$.
Then $\phi_j\td\eta_j=0$, and $\{\td\eta_j\}$ is contained in the standard sphere $\mathbb S$. The compactness of $\mathbb S$ assures that, by passing to the subsequence if possible, we can assume that $\td\eta_j\ra \td\eta_0\in {\mathbb S}$. Consequently, we have $\td\eta_0=\lim_{j\to+\infty}\phi_j\td\eta_j=0$ which is not possible! So there must me a large $R_0>0$ such that $\ker\Phi_{R_0}=0$ and the claim is proved.

For any $R\geq R_0$, it is easily seen that $\ker\Phi_R\subset\ker\Phi_{R_0}$ which implies that $\ker\Phi_R=0$ and $\phi_RV\cong V$. In particular, $\dim \phi_RV=\dim V$.

That $\dim \phi_RV^\bot_1=\dim V^\bot_1$ follows in the same way.

Next we are to find a larger $R\geq R_0$ such that $Q$ is negative definite on $\phi_RV$. For this, we first note that $|\nabla\phi_R|$ supports in $B_{2R}(o)\bsl B_R(o)$ and $|\nabla\phi_R|\leq \fr1R$, and then use Lemma \ref{lem6.3} to conclude that, for all $\eta\in{\mathbb S}$
\begin{align*}
Q(\phi_R\eta,\phi_R\eta)
\leq& -\int_M\phi_R^2 |\eta|^2 e^{-f}dV+\int_M|\nabla\phi_R|^2(|\eta|^2+|v^\top|^2)e^{-f}dV\\
\leq& -\int_M\phi_R^2 |\eta|^2 e^{-f}dV+\fr1{R^2}\int_{B_{2R}(0)\bsl B_R(0)}(|\eta|^2+|v^\top|^2)e^{-f}dV.
\end{align*}
Note we only care about $v^\bot$ here and there is nothing to do with $v^\top$. Therefore, by \eqref{7.9-1}--\eqref{7.9-3} and Lemma \ref{R0}, there must be an $R_0$ large enough such that $Q(\phi_R\eta,\phi_R\eta)<0$ for all $\eta\in{\mathbb S}$, $R\geq R_0$. Then the conclusion that $Q$ is negative definite on $\phi_RV$ follows directly from the bi-linearity of $Q$.
\end{proof}

\begin{lem}\label{v1=0} Under the complete and non-compact assumption, we have
\be
V^\bot_1=0 \text{ or equivalently }V^\bot=V^\bot_0.
\ee
\end{lem}

\begin{proof} Let $W^\bot$ be the orthogonal complement of $W$ in the space $\Gamma^{\infty,2}_w(T^\bot M^m)$ of $L^2_w$-smooth normal sections. For any given $R>0$, define a subspace
$$W^\bot(\phi_RV):=W^\bot\cap (\phi_RV)$$
of $W^\bot$ and a linear map $\Psi_R:\phi_RV^\bot_1\to W^\bot(\phi_RV)$ by
$$
\phi_Rv^\bot\mapsto \Psi_R(\phi_Rv^\bot):=\phi_Rv^\bot-\fr{\int_M\lagl\phi_R v^\bot,e_\alpha\ragl e^{-f}dV}{\int_M\phi_R e^{-f}dV}\phi_R e_\alpha,\quad\,\forall v^\bot\in V^\bot_1.
$$

Claim: There must be a large $R>0$ such that $\ker\Psi_R=0$.

In fact, if this is not true, then we can find a sequence $\{v^\bot_j\}\subset V^\bot_1$ with $\phi_jv^\bot_j\neq 0$ and $\Psi_j(\phi_jv^\bot_j)=0$ for each $j=1,2,\cdots$. It follows that $v^\bot_j\neq 0$, $j=1,2,\cdots$. Define
$$\td v^\bot_j:=\fr{v^\bot_j}{\|v^\bot_j\|_{2,w}},\quad j=1,2,\cdots.$$
Then $\Psi_j(\phi_j\td v^\bot_j)=0$, $j=1,2,\cdots$. Without loss of generality, we can assume that $\td v^\bot_j\to \td v^\bot_0$. Then $\td v^\bot_0\in V^\bot_1$ and $\|\td v^\bot_0\|_{2,w}=1$.

On the other hand, from $\Psi_j(\phi_j\td v^\bot_j)=0$ ($j=1,2,\cdots)$ it follows that
$$
\phi_j\td v^\bot_j=\fr{\int_M\lagl\phi_j\td v^\bot_j,e_\alpha\ragl e^{-f}dV}{\int_M\phi_j e^{-f}dV}\phi_j e_\alpha,\quad j=1,2,\cdots,
$$
implying that
\be\label{7.13}
\|\phi_j\td v^\bot_j\|^2_{2,w}=\fr{\int_M\lagl\phi_j\td v^\bot_j,e_\alpha\ragl e^{-f}dV}{\int_M\phi_j e^{-f}dV}(\phi_j e_\alpha,\phi_j\td v^\bot_j),\quad j=1,2,\cdots.
\ee
But it is clear that $\phi_j\td v^\bot_j\to \td v^\bot_0$ when $j\to +\infty$ since
\begin{align*}
\|\phi_j\td v^\bot_j-\td v^\bot_0\|_{2,w} \leq& \|\phi_j(\td v^\bot_j-\td v^\bot_0)\|_{2,w} +\|(\phi_j-1) \td v^\bot_0 \|_{2,w}\\
\leq& \|\td v^\bot_j-\td v^\bot_0\|_{2,w} +\|\phi_j-1\|_{2,w}\to 0,\quad j\to +\infty.
\end{align*}
Let $j\to+\infty$ in \eqref{7.13} then we obtain
$$\|\td v^\bot_0\|^2_{2,w}=\fr{\int_M\lagl\td v^\bot_0,e_\alpha\ragl e^{-f}dV}{\int_M e^{-f}dV}(e_\alpha,\td v^\bot_0)=0
$$
because $\td v^\bot_0\in V^\bot_1$ is orthogonal to $W$, contradicting to the fact that $\|\td v^\bot_0\|_{2,w}=1$. So the claim is proved.

Thus by \eqref{7.11}, when $R$ large enough it holds that
$$\dim V^\bot_1=\dim \phi_RV^\bot_1\leq \dim W^\bot(\phi_RV)\leq {\rm ind}_W(Q)$$
where ${\rm ind}_W(Q)$ denotes the $W$-stability index of $Q$. By the $W$-stability of $x$ we have ${\rm ind}_W(Q)=0$, implying that $\dim V^\bot_1=0$ and thus $V^\bot_1=0$ or equivalently $V^\bot=V^\bot_0$.
\end{proof}

{\em Proof of Theorem \ref{main6}}

Using Proposition \ref{cmpt}, we conclude that $x:M^m\to\bbr^{m+p}$ must be non-compact. Then by Lemma \ref{v1=0}, we have a direct decomposition
$$
\bbr^{m+p}=V^\top\oplus V^\bot
$$
where $V^\top$ now consists of all constant vectors in $\bbr^{m+p}$ that are tangent to $M^m$ at each point of $M^m$, while $V^\bot$ consists of all constant vectors in $\bbr^{m+p}$ that are normal to $M^m$ at each point of $M^m$. It then follows that $\dim V^\top\leq m$ and $\dim V^\bot\leq p$. Consequently
$$
m+p=\dim\bbr^{m+p}=\dim V^\top+\dim V^\bot\leq m+p
$$
which implies that $\dim V^\top=m$ and $\dim V^\bot=p$. This is true only if $M^m\equiv P^m$.

Theorem \ref{main6} is proved.

\end{document}